\documentclass{article}

\usepackage{CAAI_preprint}

\title{\textbf{Dynamics of a modified Leslie--Gower predator-prey model with Allee effect on the prey and a generalist predator}}

\author{
Claudio Arancibia--Ibarra \\
 School of Mathematical Sciences, Queensland University of Technology\\
GPO Box 2434, GP Campus, Brisbane, Queensland 4001 Australia\\
 Facultad de Ingeniería y Negocios, Universidad de Las Américas \\
 Av. Manuel Montt 948, Santiago, Chile \\
  \texttt{claudio.arancibia@hdr.qut.edu.au} \\
  \AND
Jos\'e Flores \\
Department of Mathematics, The University of South Dakota (USD)\\
Vermillion, South Dakota, USA\\
 \texttt{Jose.Flores@usd.edu} \\  
}

\begin{document}
\maketitle
\begin{abstract} \normalsize

A predator-prey model with functional response Holling type II, Allee effect in the prey and a generalist predator is considered. It is shown that the model with strong Allee effect has at most two positive equilibrium point in the first quadrant, one is always a saddle point and the other exhibits multi-stability phenomenon since the equilibrium point can be stable or unstable. While the model with weak Allee effect has at most three positive equilibrium point in the first quadrant, one is always a saddle point and the other two can be stable or unstable node. In addition, when the parameters vary in a small neighbourhood of system parameters the model undergoes to different bifurcations, such as saddle-node, Hopf and Bogadonov--Takens bifurcations. Moreover, numerical simulation is used to illustrate the impact in the stability of positive equilibrium point(s) by adding an Allee effect and an alternative food source for predators.

\end{abstract}


\section{Introduction}

In this manuscript a modified Leslie--Gower predator-prey model proposed by Leslie and Gower~\cite{leslie2} and modified by May~\cite{may} is studied. The model is described by an autonomous two-dimensional system of ordinary differential equation where the equations for the predator and the prey are a logistic-type growth function~\cite{aguirre1, flores2, turchin, qiao}, the functional response is a Holling type II functional response, sometimes called a hyperbolic functional response. The Holling type II functional response occurs in species when the number of prey consumed rises rapidly at the same time as the prey density increases~\cite{dale}. This functional response is represented by $qx/\left(x+a\right)$ where the parameter $q$ is the maximum predation rate per capita and $a$ is the population value at which the predator function is one half of the saturated level. The following pair of equations is a general representation of the model where $x\left(t\right)$ is used to represent the size of the prey population at time $t$, and $y\left(t\right)$ is used to represent the size of the predator population at time $t$; 
\begin{equation}\label{eq1}
\begin{aligned}
\dfrac{dx}{dt} &=	 rx\left(1-\dfrac{x}{K} \right)  - \dfrac{qxy}{x+a}\,,  \\
\dfrac{dy}{dt} &=	 sy\left(1-\dfrac{y}{nx} \right) \,. 
\end{aligned}
\end{equation}
Here $r$ and $s$ are the intrinsic growth rate for the prey and predator respectively, $n$ is a measure of the quality of the prey as food for the predator and $K$ is the prey environmental carrying capacity. 

Predator-prey model studied in~\cite{arancibia,aziz,feng,singh} are known as modified Leslie--Gower models. In these models the predator acts as a generalist since it avoids extinction by utilising an alternative source of food. In the case of severe prey scarcity, some predator species can switch to another available food, although its population growth may still be limited by the fact that its preferred food is not available in abundance. The Leslie--Gower predator-prey model also considers the case of a specialist predator~\cite{korobei}. It is assumed that a reduction in a predator population has a reciprocal relationship with per capita availability of its favourite food~\cite{aziz}. Nevertheless, when the alternative food is positive the modified Leslie--Gower model does not have these abnormalities and it enhances the predictions about the interactions. This model was proposed in~\cite{aziz}, but the model was only analysed partially. Using a Lyapunov function~\cite{korobei}, the global stability of a unique positive equilibrium point was shown. The alternative food for predator can be modelled by adding a positive constant $c$ to the environmental carrying capacity for the predator~\cite{aziz}. Therefore, I have a modification to the logistic growth term in the predator equation, namely $nx$ is replaced in~\eqref{eq1} by $nx+c$.

On the other hand, additional complexity can be incorporated by consider, for example, Allee effect. The Allee effect is defined as the relation between population size and fitness. The lower the population size, the lower the fitness~\cite{allee,berec,stephens}. This phenomenon has also been called depensation in fisheries sciences and indicates a positive density dependence in population dynamics~\cite{liermann}. The Allee effect may appear due to a wide range of biological phenomena, such as reduced anti-predator vigilance, social thermo-regulation, genetic drift, mating difficulty, reduced defence against the predator, and deficient feeding because of low population densities~\cite{stephens2}. The influence of the Allee effect upon the logistic-type growth in the prey equation is represented by the inclusion of a multiplier in the form of $x-m$ where $m$ is the minimum viable population or Allee threshold. For $0<m<K$, the per-capita grow rate of the prey population with the Allee effect included is negative, but increasing, for $N\in[0,m)$, and this is referred to as the strong Allee effect.  When $m\leq0$, the per-capita growth rate is positive but increases at low prey population densities and this is referred to as the weak\footnote{Note that $m=0$ is often also called the weak Allee effect, however, the behaviour of the model with $m=0$ is similar to the case of a strong Allee effect ($m>0$) and thus I will not consider the case of $m=0$ in this manuscript.} Allee effect~\cite{berec,courchamp2}. Additionally, the Allee effect can also refer to a decrease in the per-capita fertility rate at low population densities or a phenomenon in which fitness, or population growth, increases as population density increases~\cite{stephens, allee2, courchamp, kramer}.

The aim of this manuscript is to study the dynamic of the Leslie--Gower predator-prey model with functional response Holling type II, Allee effect on the prey and alternative food for predator. System~\eqref{eq2} was partially studied in~\cite{arancibia} in which the authors studied only the stability of the positive equilibrium point when the strong Allee effect on the prey. This manuscript extend the properties of the model proposed in~\cite{arancibia,arancibia8,arancibia3,gonzalez3,arancibia5,arancibia9}. Showing that while the different modifications of predator-prey models often lead to qualitatively similar dynamics, these models are sensitive to changes in their parameters. It has also shown that small changes to similar parameters in different models lead to different behaviours. For instance, the Leslie--Gower predator-prey model with strong Allee effect ($m>0$)~\cite{arancibia3}, weak Allee effect ($m=0$)~\cite{gonzalez3}, alternative food for the predator~\cite{arancibia5,arancibia9} and the model with both modifications at once~\cite{arancibia}, i.e. Allee effect and alternative food, can support the extinction and coexistence of the predator and/or the prey populations.

The modified model with both modifications at once is briefly described in Section~\ref{S1}. In Section~\ref{S2} I study the main properties of the Leslie--Gower predator-prey model with alternative food for the predator and strong Allee effect. I prove the stability of the equilibrium points and giving the conditions for a different type of bifurcations. While the case of weak Allee effect is studied in Section~\ref{S3}. Finally, in Section~\ref{S4} I discuss the results and giving the ecological association.

\section{The Model}\label{S1}
When the Allee effect on the prey and the alternative food for the predator is included in system~\eqref{eq1}. It becomes 
\begin{equation}\label{eq2}
\begin{aligned}
\dfrac{dx}{dt} & = rx\left( 1-\dfrac{x}{K}\right)\left(x-m\right) -\dfrac{qxy}{x+a}\,, \\ 
\dfrac{dy}{dt} & = sy\left( 1\ -\dfrac{y}{nx+c}\right) \,.
\end{aligned} 
\end{equation}
Here all the parameters are considered positive, i.e $\left(r,K,q,a,s,n,m,c\right)\in \mathbb{R}{}{}_{+}^{8}$, $a<K$ and system~\eqref{eq1} is a Kolmogorov type~\cite{freedman}. That is the coordinates axis are invariant and  is defined in the first quadrant $\Omega=\{\left(x,y\right)\in\mathbb{R}^2, x\geq0, y\geq0\} =\mathbb{R}^+_0 \times \mathbb{R}^+_0$.

The equilibrium points of the system~\eqref{eq2} with strong Allee effect ($m>0$) are $\left(K,0\right)$, $\left(m,0\right)$, $\left(0,0\right)$, $\left(0,c\right)$ and $\left(x^*,y^*\right)$ which is the intersection of the nullclines $$y=nx+c \quad \text{and} \quad y=\dfrac{r}{q}\left( 1-\dfrac{x}{K}\right)\left(x+a\right)\left(x-m\right).$$
Note that if $m<0$, then the equilibrium point $\left(m,0\right)$ is located on the negative half axis $x\leq0$.

I follow~\cite{blows, gonzalezy2, saez} and convert system~\eqref{eq2} to a  topologically equivalent model in order to simplify the analysis, 
\begin{equation}\label{eq3}
\begin{aligned}
\dfrac{du}{d\tau} & = u\left(u+C\right)\left(\left(u+A\right)\left( 1-u\right)\left(u-M\right) -Qv\right), \\
\dfrac{dv}{d\tau} & = S\left(u+A\right)\left(u-v+C\right)v.
\end{aligned}
\end{equation}
System~\eqref{eq3} is topologically equivalent to system~\eqref{eq2} except at the singularitie $x=-a$. I introduce a change of variable and time rescaling, given by the function $\varphi :\breve{\Omega}\times\mathbb{R}\rightarrow \Omega\times\mathbb{R}$, where $\varphi\left(u,v,\tau\right)=\left(x,y,t\right)$, defined by $x=Ku$, $y=nKv$, $d\tau =rK\,dt/u\left(u+a/K\right)$ and $\breve{\Omega}=\left(u,v\right)\in \mathbb{R}^2, u\geq0, v\geq0$. System~\eqref{eq3} is obtained upon defining $A:=a/K<1$, $C:=c/Kn$, $S:=s/rK$, $Q:=nq/rK$ and $M:=m/K$, so $\left(A,M,C,S,Q\right)\in\left(0,1\right)\times\left(0,1\right)\times \mathbb{R}^3_+$. The mapping $\varphi$ is a diffeomorphism~\cite{chicone} preserving the orientation of time since $\det\,D\varphi\left(u,v,\tau\right)=nu\left(a+Ku\right)/r>0$.

As system~\eqref{eq3} is a Kolmogorov type~\cite{freedman}, then the $u$-axis and  $v$-axis are invariant sets. Additionally, if $u=1$ I have $du/d\tau=-Qv\left(1+C\right)$ and whatever it is the sign of $dv/d\tau=Sv\left(1-v+C\right)\left(1+A\right)$ the trajectories enter and remain in the region $\Gamma=\{\left(u,v\right)\in \breve{\Omega},0\leq u\leq1, v\geq0$. Additionally, from the first equation of the system~\eqref{eq2} I have that since $Qu\left(u+C\right)v>0$ and $0<u<1$. Then,  
\begin{equation}\label{mhtae4}
\dfrac{du}{d\tau}\leq\left(1+A\right)\left(1+C\right)\left(1-u\right)\left(u-M\right).
\end{equation}
Separating variables and integrating both side of~\eqref{mhtae4}, I have
\[\begin{aligned}
\ln\left(u-M\right)-\ln\left(u-1\right)&<\left(1+A\right)\left(1+C\right)\tau+\varsigma,\\
u &< \dfrac{M-e^{\left(A+1\right)\left(1-M\right)\left(1+C\right)\tau+\varsigma}}{1-e^{\left(A+1\right)\left(1-M\right)\left(1+C\right)\tau+\varsigma}}=\dfrac{Me^{-\left(A+1\right)\left(1-M\right)\left(1+C\right)\tau+\varsigma}-e^{\varsigma}}{e^{-\left(A+1\right)\left(1-M\right)\left(1+C\right)\tau+\varsigma}-e^{\varsigma}}.
\end{aligned}\]
Where $\varsigma$ depends on the initial conditions. In addition $\limsup\limits_{\tau \rightarrow \infty} u\left(\tau\right)\leq1$.
	
On the other hand, from the second equation from the system~\eqref{eq3} I have that since $0<u\left(\tau\right)<1$ and $v>0$; thus $u-v+C<1-v+C$ for all $v\geq0$ and $Sv\left(u-v+C\right)<Sv\left(1-v+C\right)$. In addition $u+A<1+A$, therefore $Sv\left(u-v+C\right)\left(u+A\right)<Sv\left(1-v+C\right)\left(1+A\right)$. Thus
	\begin{equation}\label{mhtae5}
	\dfrac{dv}{d\tau}\leq Sv\left(1+A\right)\left(1-v+C\right).
	\end{equation}
	Separating variables and integrate both side of~\eqref{mhtae5}, I have 
	\[\begin{aligned}
	\ln\left(v\right)-\ln\left(v-1-C\right)&<S\left(1+A\right)\left(1+C\right)\tau+\varsigma,\\
	v&<\dfrac{-\left(1+C\right)e^{S\left(1+A\right)\left(1+C\right)\tau+\varsigma}}{1- e^{S\left(1+A\right)\left(1+C\right)\tau+\varsigma}}=\dfrac{\left(1+C\right)e^{\varsigma}}{e^{\varsigma}-e^{-S\left(1+A\right)\left(1+C\right)\tau}}.
	\end{aligned}\]
	Where $\varsigma$ depends on the initial conditions. In addition, $\limsup\limits_{\tau \rightarrow \infty} v\left(\tau\right)\leq1+C$.
	
	Next, I will show that $u\left(\tau\right)+v\left(\tau\right)$ are bounded. Set the function $w\left(\tau\right)=u\left(\tau\right)+v\left(\tau\right)$, then
	\[\dfrac{dw}{d\tau}=\dfrac{du}{d\tau}+\dfrac{dv}{d\tau}\leq\left(u+A\right)\left(1-u\right)\left(u-M\right)+Sv\left(1+A\right)\left(1-v+C\right).\]
	Thus, since $u\leq1$, $v\leq1$ and using the maximum value of $\left(u+A\right)\left(1-u\right)\left(u+C\right)$ and  $Sv\left(1+A\right)\left(1-v+C\right)$, I got
	\[\dfrac{dw}{d\tau}+w\left(\tau\right)\leq \dfrac{\left(S\left(1+A\right)^2\left(1-M\right)^2\left(1+C\right)+4S\left(1+A\right)+\left(S+AS+CS+ACS+1\right)^2 \right) }{4S\left(A+1\right)}\stackrel{\triangle}{=}\Theta.\] 
	Using the theory of differential inequality, for all $0\leq \tau_0\leq \tau$, I have that
	\begin{equation}\label{mhtae6}
	\dfrac{dw}{d\tau}+w\left(\tau\right)\leq \Theta.
	\end{equation}
	Multyplying both sides of~\eqref{mhtae6} by $I\left(\tau\right)$, a positive  integrating factor function 
	\[I\left(\tau\right)\dfrac{dw}{d\tau}+I\left(\tau\right)w\left(\tau\right)\leq I\left(\tau\right)\Theta \quad \text{with}\quad I\left(\tau\right)=e^{\int\limits_{\tau_0}^{\tau} d\tau}=e^{\tau-\tau_0}.\] \\
	So,  $w\left(\tau\right)\leq \Theta+\left(w\left(\tau_0\right)-\Theta\right)/\left(e^{\tau-\tau_0}\right)$. Hence, $\limsup\limits_{\tau \rightarrow \infty} w\left(\tau\right)\leq \Theta$. Therefore, all solutions in system~\eqref{eq3} with strong and weak Allee effect which are initiated in $\mathbb{R}^2_+$ and with initial values $u\left(0\right)=u_0>0$ and $v\left(0\right)=v_0>0$ are bounded. 

The $u$ nullclines for system~\eqref{eq3} are $u=0$ and $v=\left(u+A\right)\left(1-u\right)\left(u-M\right)/Q$, while the $v$ nullclines are $v=0$ and $v=u+C$. The equilibrium points for the system~\eqref{eq3} with strong Allee effect, i.e $M>0$ are $\left(0,0\right)$, $\left(1,0\right)$, $\left(M,0\right)$ (strong Allee effect), $\left(0,C\right)$ and the point(s) $\left(u^*,v^*\right)$, where $u^*$ is determined by the solution of
\begin{align}
& \left(u+A\right)\left(1-u\right)\left(u-M\right)/Q=u+C \,, \quad\text{or equivalently}\,,\nonumber\\
& f\left(u\right)=u^3-\left(M+1-A\right)u^2-\left(A\left(M+1\right)-Q-M\right)u+AM+CQ=0\,.\label{eq4}
\end{align}

\section{Strong Allee effect ($M>0$)}\label{S2}
Next, I study the case of $M>0$, then the cubic function $g\left(u\right)=\left(u+A\right)\left(1-u\right)\left(u-M\right)/Q$ always intersect the straight line $h\left(u\right)=u+C$ in one point; which is located in the second or third quadrant (see Figure~\ref{Fig1}). So there will be always a single negative real root, which I denote by $u=-G$.
\begin{figure}
\centering
\includegraphics[width=16cm]{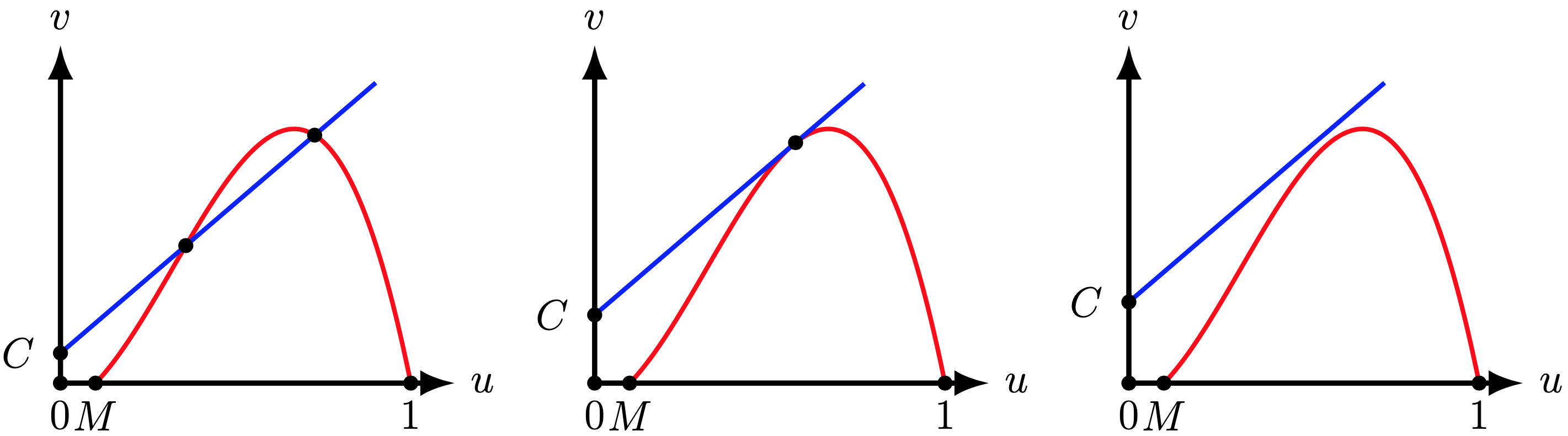}
\caption{Intersection of the function $g\left(u\right)=\left(u+A\right)\left(1-u\right)\left(u-M\right)/Q$ (red) and the straight line $h\left(u\right)=u+C$ (blue) for system~\eqref{eq3} affected by strong Allee effect, i.e $M>0$.}
\label{Fig1}
\end{figure}
Due to the difficult to determine the exact solutions of equation~\eqref{eq4}, I divide the cubic equation by $u+G$, I obtain the second order polynomial
\begin{equation}\label{eq5}	
u^2-\left(1-A+G+M\right)u+\left(M+Q-A\left(M+1\right)+G\left(1-A+G+M\right)\right)=0\,,
\end{equation} 
From~\eqref{eq4}, I get that $Q=\left(G+1\right)\left(G+M\right)\left(A-G\right)/\left(C-G\right)$, and since $Q>0$, then I obtain that $A<G<C$ or else, $A>G>C$.

The roots of~\eqref{eq5} are given by 
\begin{equation}\label{eq6}
\begin{aligned}
&u_{1,2} = \dfrac{1}{2}\left(1-A+G+M \pm \sqrt{\Delta}\right)\\
&\text{with}~\Delta=\left(1-A+G+M\right)^2-4\left(M+Q-A\left(M+1\right)+G\left(1-A+G+M\right)\right).
\end{aligned}
\end{equation}
Therefore, the solutions of the equation~\eqref{eq5} depend on the value of $\Delta$~\eqref{eq6} and thus:
\begin{itemize}
	\item if $\Delta<0$~\eqref{eq6}, then system~\eqref{eq2} has no positive equilibrium point in the first quadrant (see Region IV in Figure~\ref{Fig5});
	\item if $\Delta>0$~\eqref{eq6}, then system~\eqref{eq2} has two positive equilibrium points in the first quadrant $P_1=\left(u_1,u_1+C\right)$ and $P_2=\left(u_2,u_2+C\right)$ (see Region I, II and III in Figure~\ref{Fig5}); and
	\item if $\Delta=0$~\eqref{eq6}, then system~\eqref{eq2} has one positive equilibrium point in the first quadrant $P_1=P_2=\left(E,E+C\right)$ (see $Q=Q^*$ in Figure~\ref{Fig5}). With $E=\left(1-A+G+M\right)/2$.
\end{itemize}
Note that if $\Delta=0$~\eqref{eq6}, then two positive equilibrium points collapses, i.e $P_1=P_2$. I also observe that none of these equilibrium points explicitly depend on the system parameter $S$. Therefore, $S$ and $Q$ are the natural candidates to act as bifurcation parameters. 
 
\subsection{Nature of equilibrium points}
To determine the nature of the equilibrium points I must compute the Jacobian matrix $J\left(u,v\right)$ of system~\eqref{eq3} with strong Allee effect, that is:
\begin{equation}\label{eq7}
J\left(u,v\right)=\begin{pmatrix}
-5u^4+4\left(M-C-A+1\right)u^3+\beta  & -Qu\left(u+C\right) \\ 
Sv\left(A+C+2u-v\right)  &  S\left(C+u-2v\right)\left(A+u\right) 
\end{pmatrix}.
\end{equation}
With $\beta=3\left(A+C-M-AC+AM+CM\right)u^2+2\left(AC-AM-CM-Qv+ACM\right)u-C\left(AM+Qv\right)$.
\begin{lemm}
The equilibrium points $\left(0,0\right)$ and $\left(1,0\right)$ are a saddle point, $\left(M,0\right)$ is a repeller point and $\left(0,C\right)$ is an attractor point.
\end{lemm}
\begin{proof}
The Jacobian matrix~\eqref{eq7} evaluate at the equilibrium point $\left(0,0\right)$ gives
\[ J\left(0,0\right)=\begin{pmatrix} -ACM  & 0 \\ 0  &  ACS \end{pmatrix}.\]
Hence, $\det\left(J\left(0,0\right)\right)=-A^2C^2MS<0$ since $0<M<1$. Therefore, the equilibrium $\left(0,0\right)$ is a saddle point.
Similarly, the Jacobian matrix~\eqref{eq7} evaluate at the equilibrium point $\left(1,0\right)$ gives
\[ J\left(1,0\right)=\begin{pmatrix} -\left(1-M\right)\left(C+1\right)\left(A+1\right)  & -Q\left(C+1\right) \\  0  &  S\left(C+1\right)\left(A+1\right) \end{pmatrix}.\]
Hence, $\det\left(J\left(1,0\right)\right)=-S\left(1-M\right)\left(C+1\right)^2\left(A+1\right)^2<0$ since $0<M<1$. Therefore, the equilibrium $\left(1,0\right)$ is also a saddle point.
The Jacobian matrix~\eqref{eq7} evaluate at the equilibrium point $\left(M,0\right)$ gives
\[J\left(M,0\right)=\begin{pmatrix} M\left(1-M\right)\left(C+M\right)\left(A+M\right)  & -MQ\left(C+M\right) \\  0  & S\left(A+M\right)\left(C+M\right) \end{pmatrix}.\]
Hence, $\det\left(J\left(M,0\right)\right)=MS\left(1-M\right)\left(C+M\right)^2\left(A+M\right)^2>0$ and $\tr\left(J\left(M,0\right)\right)=M\left(C+M\right)\left(A+M\right)\left(1-M+S\right)>0$ since $0<M<1$. Therefore, the equilibrium $\left(M,0\right)$ is a repeller point.
Finally, the Jacobian matrix~\eqref{eq7} evaluate at the equilibrium point $\left(0,C\right)$ gives
\[J\left(0,C\right)=\begin{pmatrix} -C\left(AM+QC\right)  & 0 \\  ACS  &  -ACS  \end{pmatrix}.\]
Hence, $\det\left(J\left(0,C\right)\right)=AC^2S\left(AM+QC\right)>0$ and $\tr\left(J\left(0,C\right)\right)=-C\left(AM+QC+AS\right)<0$ since $0<M<1$. Therefore, the equilibrium $\left(0,C\right)$ is an attractor point.
\end{proof}

The positive singularity lie on the curve $u=v+C$. So that, the Jacobian matrix of the system~\eqref{eq3} with strong Allee effect is:
\begin{equation} \label{eq8}
J\left(u,u+C\right)=\begin{pmatrix}
 u\left(u+C\right)(-M+A(1+M)+2u(M+1-A)-3u^2)  & -Qu \left(u+C\right)\\ 
S\left(u+A\right) \left(u+C\right) &  -S\left(u+A\right)\left(u+C\right) 
\end{pmatrix}.
\end{equation}
The determinant and the trace of the Jacobian matrix~\eqref{eq8} are:
\begin{align}
\det\left(J\left(u,u+C\right)\right) =& Su\left(A+u\right)\left(C+u\right)^2\left(Q-J_{11}\left(u\right)\right)	\,,\label{eq9}\\
\tr\left(J\left(u,u+C\right)\right) 	=& \left(C+u\right)\left(uJ_{11}-S\left(A+u\right)\right)\,.\label{eq10}
\end{align}
With $J_{11}=M-A(1+M)-2u+2Au-2Mu+3u^2$. Note that the signs of the determinant~\eqref{eq9} depends on the value of $Q-J_{11}\left(u\right)$ and the signs of the trace~\eqref{eq10} depends on the value of $uJ_{11}\left(u\right)-S\left(A+u\right)$, see Figure~\ref{Fig2}.
\begin{theo}\label{p1}
Let the system parameters of~\eqref{eq3} be such that $M>0$ and $\Delta>0$. Then the equilibrium point $P_{1}$ is a saddle point.
\end{theo}
\begin{proof}
Evaluating $Q-J_{11}\left(u\right)$~\eqref{eq9} at $u_1$ gives
	\[Q-J_{11}\left(u_1\right)=-\sqrt{\Delta}\left(1-A+3G+M-\sqrt{\Delta}\right)<0.\]
Hence, $\det\left(J\left(P_{1}\right)\right)<0$ and thus $P_{1}$ is a saddle point.
\end{proof}
\begin{theo}\label{p2}
Let the system parameters of~\eqref{eq3} be such that $M>0$ and $\Delta>0$. Then the equilibrium point $P_{2}$ is
\begin{enumerate}
\item stable if $S<\dfrac{J_{11}\left(u_1\right)\left(1-A+M+G+\sqrt{\Delta}\right)}{4\left(1+A+M+G+\sqrt{\Delta}\right)}$,
\item unstable if $S>\dfrac{J_{11}\left(u_1\right)\left(1-A+M+G+\sqrt{\Delta}\right)}{4\left(1+A+M+G+\sqrt{\Delta}\right)}$,
\item a centre if $S=\dfrac{J_{11}\left(u_1\right)\left(1-A+M+G+\sqrt{\Delta}\right)}{4\left(1+A+M+G+\sqrt{\Delta}\right)}$.
\end{enumerate}
\end{theo}
\begin{proof}
Evaluating $Q-J_{11}\left(u\right)$~\eqref{eq9} at $u_2$ gives
	\[Q-J_{11}\left(u_2\right)=\sqrt{\Delta}\left(1-A+3G+M-\sqrt{\Delta}\right)>0.\]
Hence, $\det\left(J\left(P_{2}\right)\right)>0$. Then, the behaviour of the equilibrium point $P_2$ depends on the trace~\eqref{eq10} of the Jacobian matrix~\eqref{eq8}. Evaluating $uJ_{11}\left(u\right)-S\left(A+u\right)$~\eqref{eq10} at $u_2$ gives
\[\begin{aligned}
uJ_{11}\left(u\right)-S\left(A+u\right)=	&\dfrac{1}{8}\left(\left(1-A+M+G+P\right)\left(1+A-M-G-P\right)\left(1-A+G+P-M\right)\right.\\
				&\left.+2\left(1+M+G+P+A\right)\left(A-G-P\right)-4S\left(1+A+M+G+P\right)\right).
\end{aligned}\]

Therefore, the behaviour of the equilibrium point depends on the value of $uJ_{11}\left(u\right)-S\left(A+u\right)$, see Figure~\ref{Fig2}.
\end{proof}
\begin{figure}
\centering
\includegraphics[width=8cm]{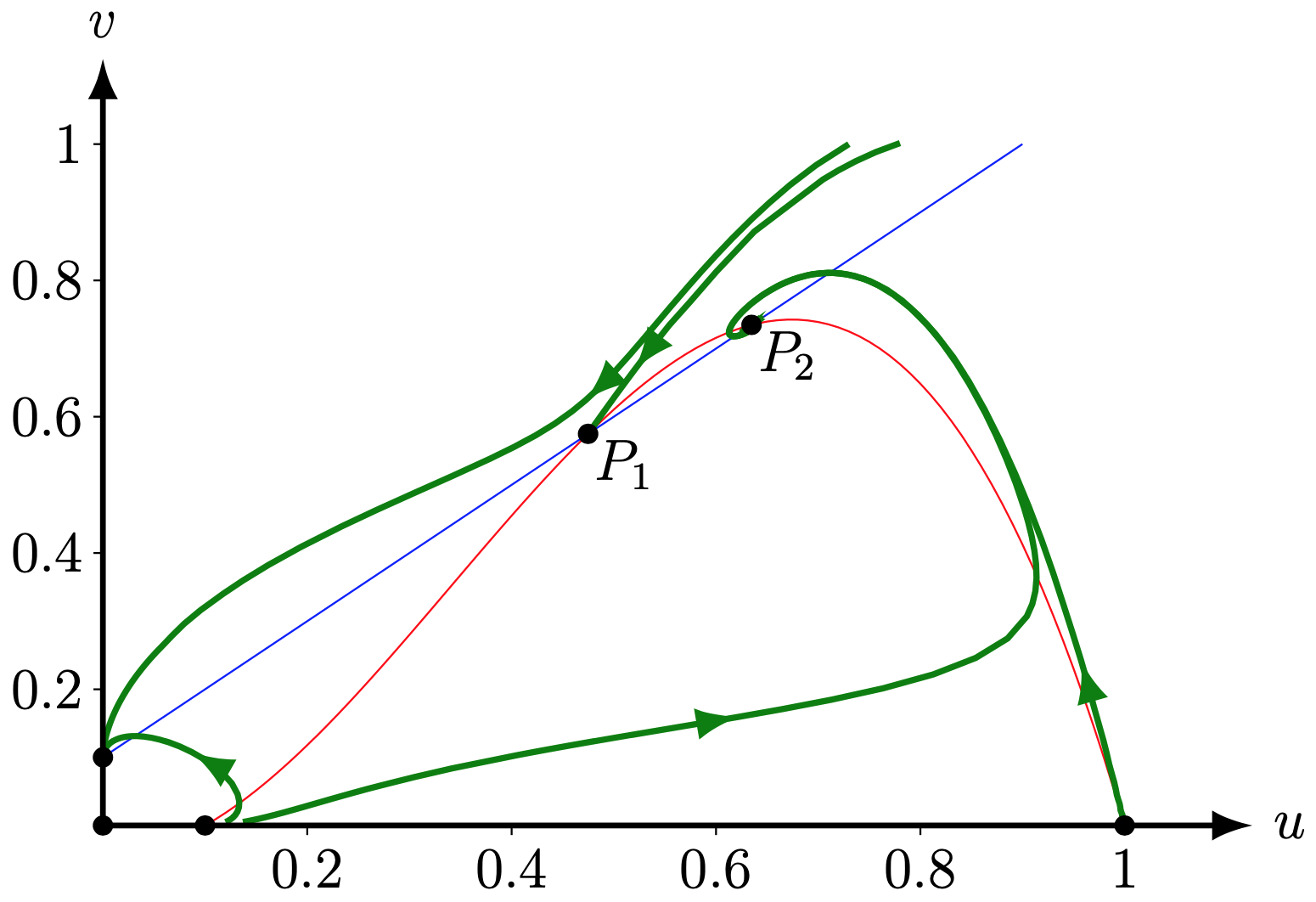}
\includegraphics[width=8cm]{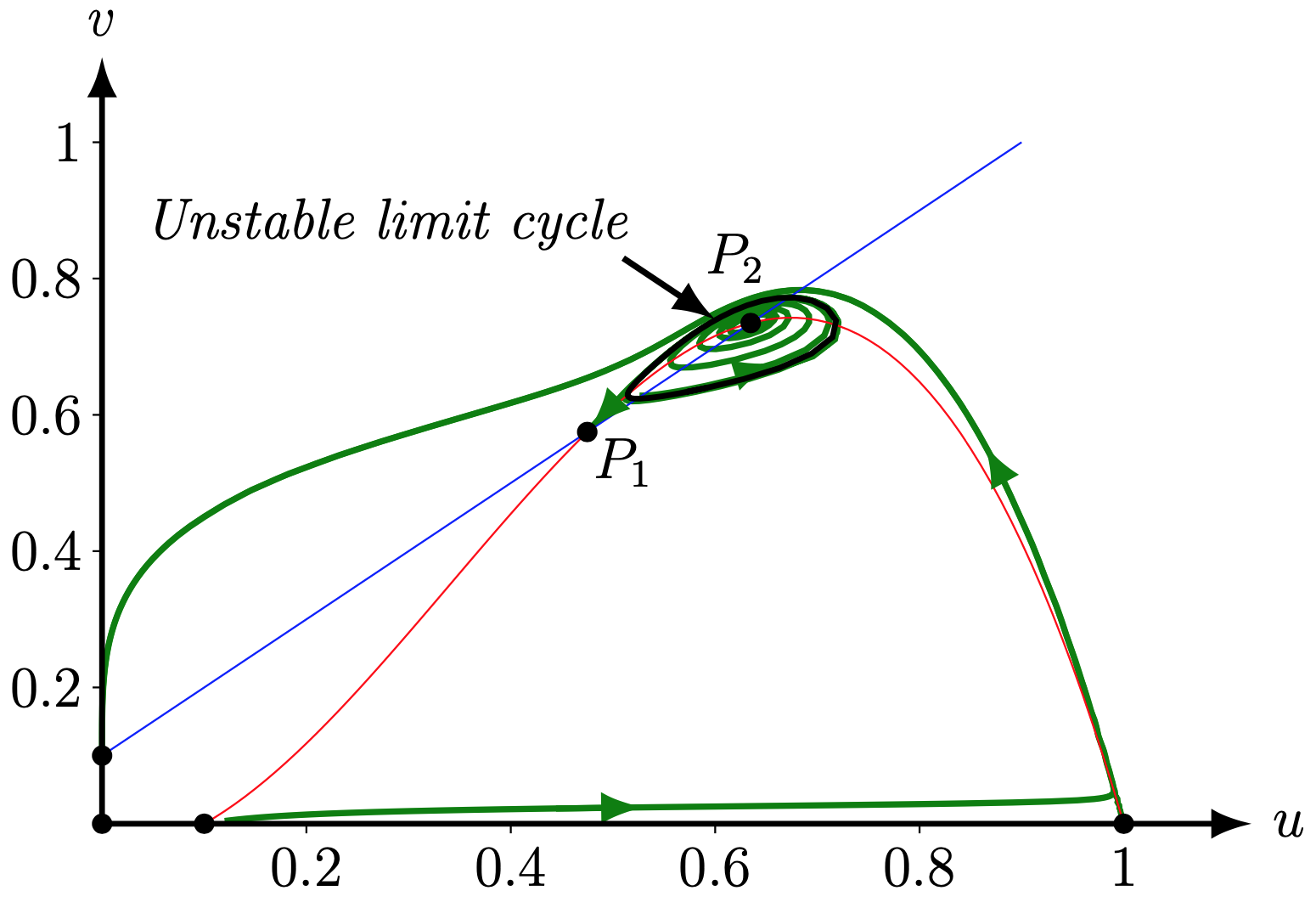}
\includegraphics[width=8cm]{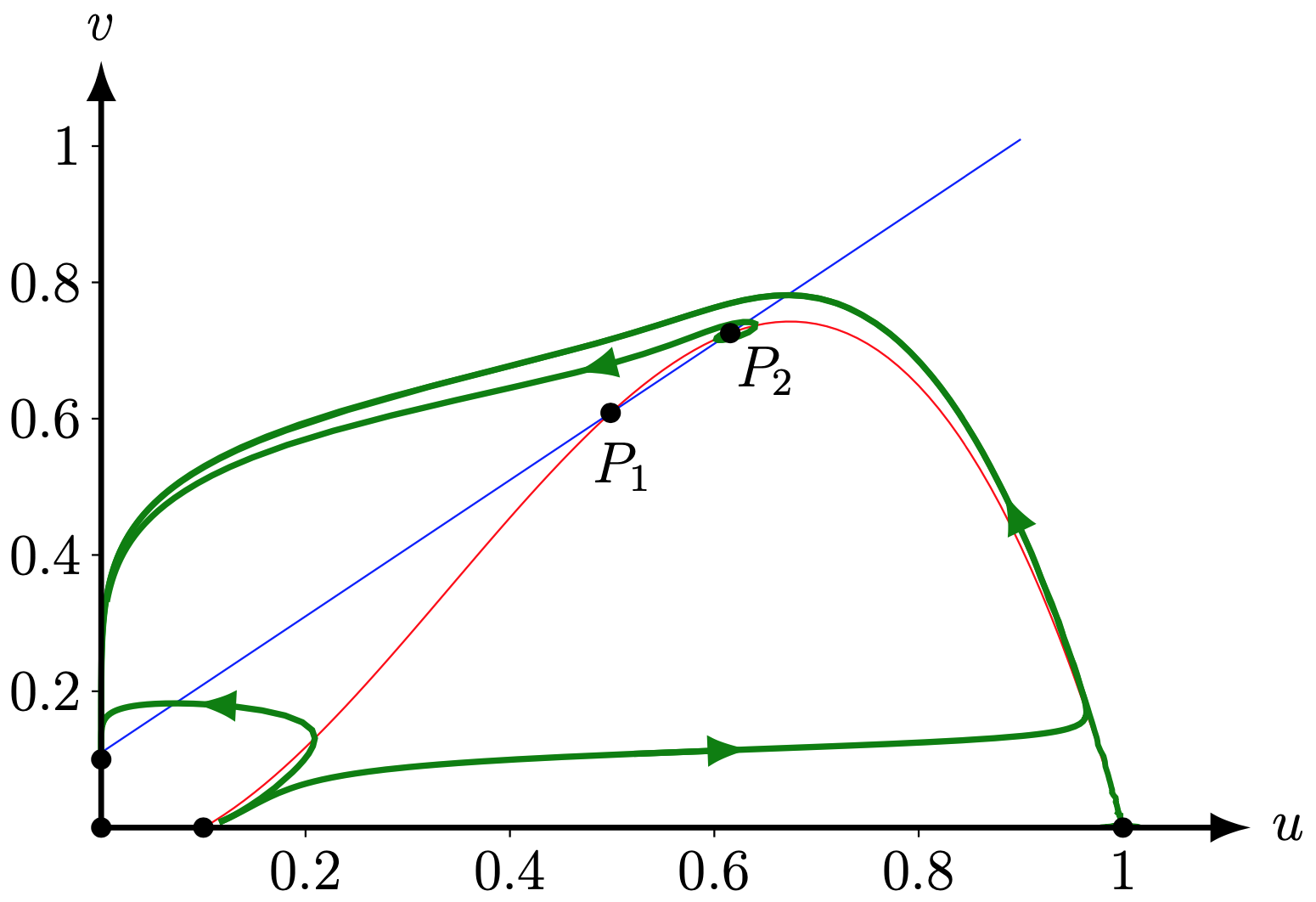}
\caption{The blue (red) curve represents the predator (prey) nullcline. If $M=0.1$; $A=0.08$; $Q=0.19$, $C=0.1$ fixed, then system~\eqref{eq3} has two positive equilibrium points namely $P_{1}$ and $P_{2}$. In the top left panel if $S=0.2$, then $P_{2}$ is a stable node. In the top right panel if $S=0.08$, then $P_{2}$ is a stable node surrounded by an unstable limit cycles. In the bottom panel if $S=0.06$, then the equilibrium point $P_{2}$ is unstable and thus the equilibrium point $\left(0,C\right)$ is a global attractor.}\label{Fig2}
\end{figure}

Note that as the trace changes sign, a Hopf bifurcation occurs~\cite{chicone} at the equilibrium point $P_{2}$. Thus, it is surrounded by an unstable limit cycle.

Let $W^s_s\left(P_{1}\right)$ be the side of the stable manifold of $P_{1}$ that goes down to the left and its acts as a separatrix curve $\Sigma$ in the first quadrant. Therefore, any initial conditions above this separatrix has the $\omega$-limit the point $\left(0,C\right)$. Moreover, any initial conditions under this separatrix has the $\omega$-limit the point $P_{2}$ when its is stable. 

\begin{theo}
There exists conditions in the system parameters for which a heteroclinic curve joining the equilibrium points $(1,0)$ and $P_{1}$. 
\end{theo}
\begin{proof}
Let $W^u_s\left(\left(1,0\right)\right)$ be the side of the unstable manifold of the saddle point $\left(1,0\right)$ that goes up to the left and $W^s_s\left(P_{1}\right)$ the side of the stable manifold of the saddle point $P_{1}$ that goes down to the left. It is clear that the curve determined by $W^u_s\left(\left(1,0\right)\right)$ remain at $\Gamma$ since it is an invariant region and its $\omega$-limit can be the point $P_2$ when it is stable or $\left(0,C\right)$ when $P_2$ is stable surrounded by an unstable limit cycle or $P_2$ is unstable.

Assuming that the $\alpha$-limit of $W^s_s\left(P_{1}\right)$ is out of $\Gamma$, then the curve $\Sigma$ is above of the curve determined by $W^u_s\left(\left(1,0\right)\right)$. If the $\alpha$-limit of $W^s_s\left(P_{1}\right)$ is inside of $\Gamma$, then the curve $\Sigma$ is below the curve determined by $W^u_s\left(\left(1,0\right)\right)$. Then, by the theorem of existence and uniqueness of solutions~\cite{chicone}, there exists a subset of system parameters for which the two manifolds coincide, forming the heteroclinic curve, see Figure~\ref{Fig3}. 
\end{proof}

\begin{theo}\label{HOM}
There exists conditions on the parameter values for which:
\begin{enumerate}
\item It exists an homoclinic curve determined by the stable and unstable manifold of point $P_{1}$,
\item It exists a non-infinitesimal limit cycle that bifurcates of the homoclinic~\cite{gaiko} surrounding the point $P_{2}$.
\end{enumerate}
\end{theo}
\begin{proof}
I observe that if $\left(u,v\right)\in\overline{P_{1}P_{2}}$, then $du/dt>0$ and thus the direction of the vector field at the points lying in the predator nullcline, i.e $v=u+C$, is to the right, since 
	\[\dfrac{du}{dt}=u\left(u+C\right)\left(\left(1-u\right)\left(u+A\right)\left(u-M\right)-Qu\right)>0\quad \text{and}\quad \dfrac{dv}{dt}=0.\]
Let $W^s_s\left(P_{1}\right)$ be the side of the stable manifold of $P_{1}$ that goes down to the left , $W^s_i\left(P_{1}\right)$ be the side of the stable manifold of $P_{1}$ that goes up to the right and $W^u_r\left(P_{1}\right)$ be the side of the unstable manifold of $P_{1}$ that goes up to the right. By the theorem of existence and uniqueness of solutions~\cite{chicone} $W^u_r\left(P_{1}\right)$ cannot intersect the trajectory determined by $W^s_s\left(P_{1}\right)$, since $\Gamma$ is an invariant region and the trajectories cannot cross the line $u=1$ towards the right. Therefore, the $\omega$-limit of $W^u_r\left(P_{1}\right)$ must be the point $P_{2}$ when it is stable, the equilibrium point $\left(0,C\right)$ when the equilibrium point $P_{2}$ is unstable or stable surrounded by an unstable limit cycle. By continuity of the system parameters I get that $W^u_r\left(P_{1}\right)$ can connects with $W^s_s\left(P_{1}\right)$ and thus a homoclinic curve is obtained, see Figure~\ref{Fig3}.
\begin{figure}
\centering
\includegraphics[width=14cm]{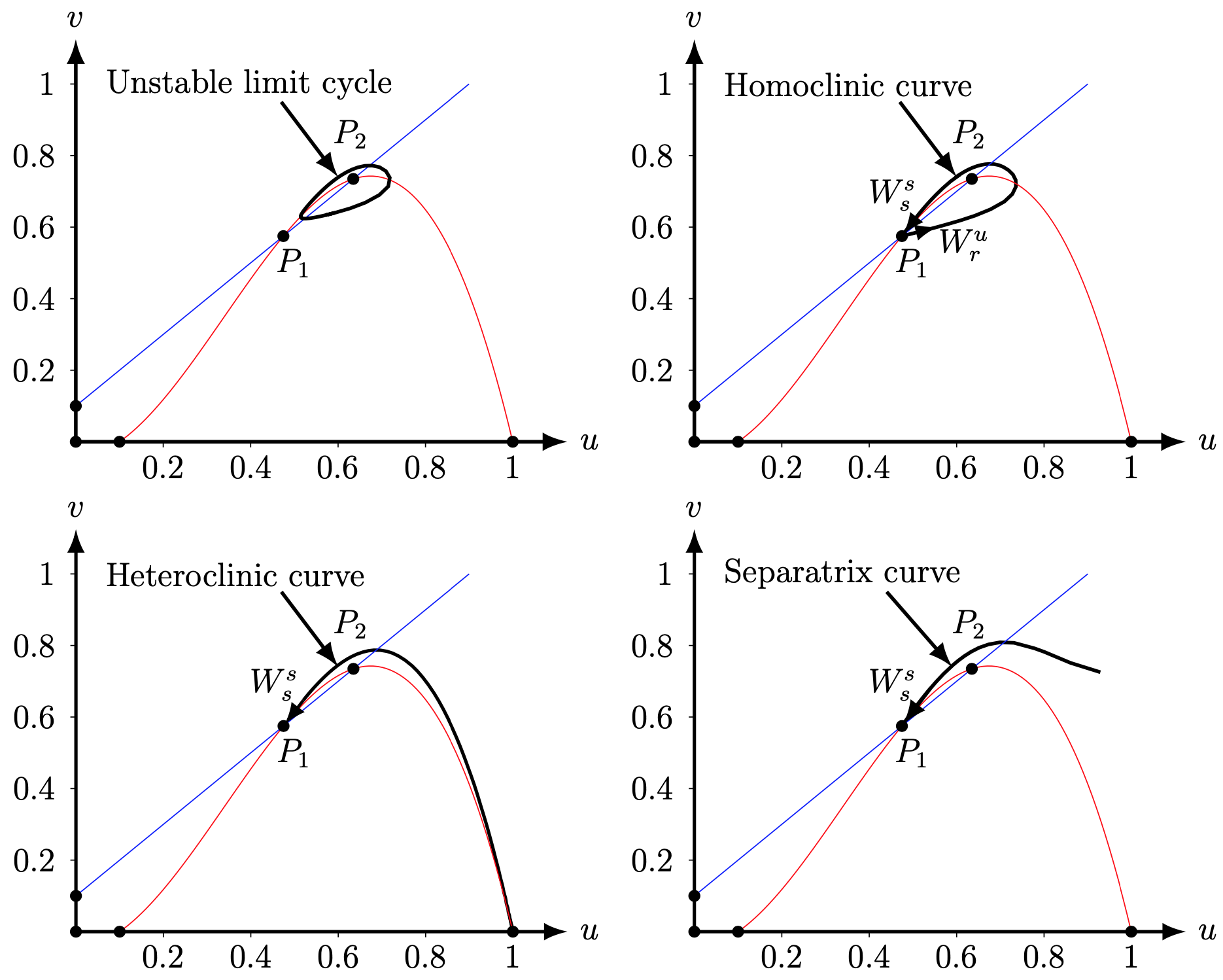}
\caption{The blue (red) curve represents the predator (prey) nullcline. If $M=0.1$, $A=0.08$,  $Q=0.19$ and $C=0.1$ are fixed, then by continuity of the parameter $S$ the unstable limit cycle increase the amplitude until its coincide with an homoclinic curve (see top panel). Then, the side of the stable manifold of the saddle point $P_{1}$ that goes down to the left ($W^s_s\left(P_{1}\right)$) connects with the side of the unstable manifold of the saddle point $\left(1,0\right)$ that goes up to the left ($W^u_s\left(\left(1,0\right)\right)$) forming an heteroclinic curve (see bottom left panel). Finally, in the bottom right panel $W^s_s\left(P_{1}\right)$ form a separatrix curve.}\label{Fig3}
\end{figure}

On the other hand, the breaking of the homoclinic curve determined by the intersection of $W^s_s\left(P_{1}\right)$ and $W^u_r\left(P_{1}\right)$, i.e $W^s_s\left(P_{1}\right)\cap W^u_r\left(P_{1}\right)$, generates a non--infinitesimal limit cycle (originating a homoclinic bifurcation), which could coincide with other limit cycle obtained via Hopf bifurcation (infinitesimal limit cycle), when  $P_{2}$ is a centre-focus, see Theorem~\ref{p2}.
\end{proof}

Next, I study the case when $\Delta=0$~\eqref{eq6}. Thus, the equilibrium points $P_1$ and $P_2$ collapse such that  $P_1=P_2=\left(E,E+C\right)$ with $E=\left(1-A+G+M\right)/2$.
\begin{theo} 
Let the system parameters be such that $\Delta=0$ and $M>0$, then the equilibrium point $\left(E,E+C\right)$ is:
\begin{enumerate}
\item a saddle-node attractor if $S<\dfrac{\left(1-A+M+G\right)\left(1+A-M-G\right)\left(1-A+G-M\right)}{4\left(1+A+M+G\right)}+\dfrac{A-G}{2},$
\item a saddle-node repeller if $S>\dfrac{\left(1-A+M+G\right)\left(1+A-M-G\right)\left(1-A+G-M\right)}{4\left(1+A+M+G\right)}+\dfrac{A-G}{2}.$
\end{enumerate}
\end{theo}
\begin{proof}
Evaluating $Q-J_{11}\left(u\right)$~\eqref{eq9} at $E$ gives
	\[Q-J_{11}\left(E\right)=0\]
	Hence, $\det\left(J\left(E,E+C\right)\right)=0$. Then, the behaviour of the equilibrium point $\left(E,E+C\right)$ depends on the trace~\eqref{eq10} of the Jacobian matrix~\eqref{eq8}. Evaluating $uJ_{11}\left(u\right)-S\left(A+u\right)$~\eqref{eq10} at $E$ gives
	\[uJ_{11}\left(E\right)-S\left(A+E\right)=\dfrac{1}{4}\left(\left(1-A+M+G\right)\left(1+A-M-G\right)\left(1-A+G-M\right)+2\left(1+M+G+A\right)\left(A-G\right)-4S\left(1+A+M+G\right)\right).\]
Therefore, the behaviour of the trace and thus the stability of the equilibrium point $\left(E,E+C\right)$ depends on the value of $uJ_{11}\left(E\right)-S\left(A+E\right)$.
\end{proof}
\begin{theo}
Let the system parameter be such that $\Delta<0$~\eqref{eq6}, then system~\eqref{eq3} has no positive equilibrium points and thus $\left(0,C\right)$ is global attractor.	
\end{theo}
\begin{proof}
I have that all solutions of system~\eqref{eq3} are bounded and $\Gamma$ is an invariant region. Moreover, the equilibrium point $\left(1,0\right)$ is a saddle point and if $\Delta<0$~\eqref{eq6} then there are no positive equilibrium points in the first quadrant. Therefore, by the Poincar\'e--Bendixon Theorem the only $\omega$--limit of the solutions in the first quadrant is the equilibrium point $\left(0,C\right)$. 
\end{proof}

\subsection{Bifurcation Analysis}
In this section I will discuss the bifurcation analysis of system~\eqref{eq3} for $\Delta=0$~\eqref{eq6} and $M>0$.
\begin{theo}\label{BT}
Let the system parameters be such that $\Delta=0$~\eqref{eq6}, $M>0$ and $Q=S\left(1+A+M+G\right)/\left(1-A+M+G\right)$, then system~\eqref{eq3}	 undergoes a Bogdanov--Takens bifurcation.
\end{theo}
\begin{proof}
	If $Q=S\left(1+A+M+G\right)/\left(1-A+M+G\right)$, then the trace is $\tr\left(J\left(E,E+C\right)\right)=0$ and the Jacobian matrix~\eqref{eq8} at the equilibrium point $\left(E,E+C\right)$ simplified to
	\[\begin{aligned}
	J\left(E,E+C\right) &=\begin{pmatrix}
	S\left(E+A\right)E & -S\left(E+A\right)E \\ 
	S\left(E+A\right)E & -S\left(E+A\right)E 
	\end{pmatrix}=\dfrac{S}{4}\left(1-A+M+G\right)\left(1+M+G+A\right) \begin{pmatrix}
	1 & -1 \\ 
	1 & -1 
	\end{pmatrix}. 
	\end{aligned}\]
	Now, I find the Jordan normal form of $J\left(E,E+C\right)$ which has equal eigenvalues and a unique eigenvector
	$\begin{pmatrix}
	1 \\ 
	1 
	\end{pmatrix}$. This vector will be the first column of the matrix of transformations $\Upsilon$. To obtain the second column I choose a vector that makes the matrix $\Upsilon$, that is $\begin{pmatrix}
	-1 \\ 0 \end{pmatrix}$. Thus,
	\[\Upsilon=\begin{pmatrix} 1 & -1 \\ 1 & 0 \end{pmatrix}~\text{and}~\Upsilon^{-1}\left(J\left(E,E+C\right)\right)\Upsilon=\begin{pmatrix}0 & \dfrac{S}{4}\left(1-A+M+G\right)\left(1+M+G+A\right) \\ 0 & 0 \end{pmatrix}.\]
	Hence, I have the Bogdanov--Takens bifurcation or bifurcation of codimension 2~\cite{xiao2}. Thus, the point $\left(E,E+C\right)$ is a cusp point, see Figure~\ref{Fig4}.
\end{proof} 
\begin{figure}
\centering
\includegraphics[width=8cm]{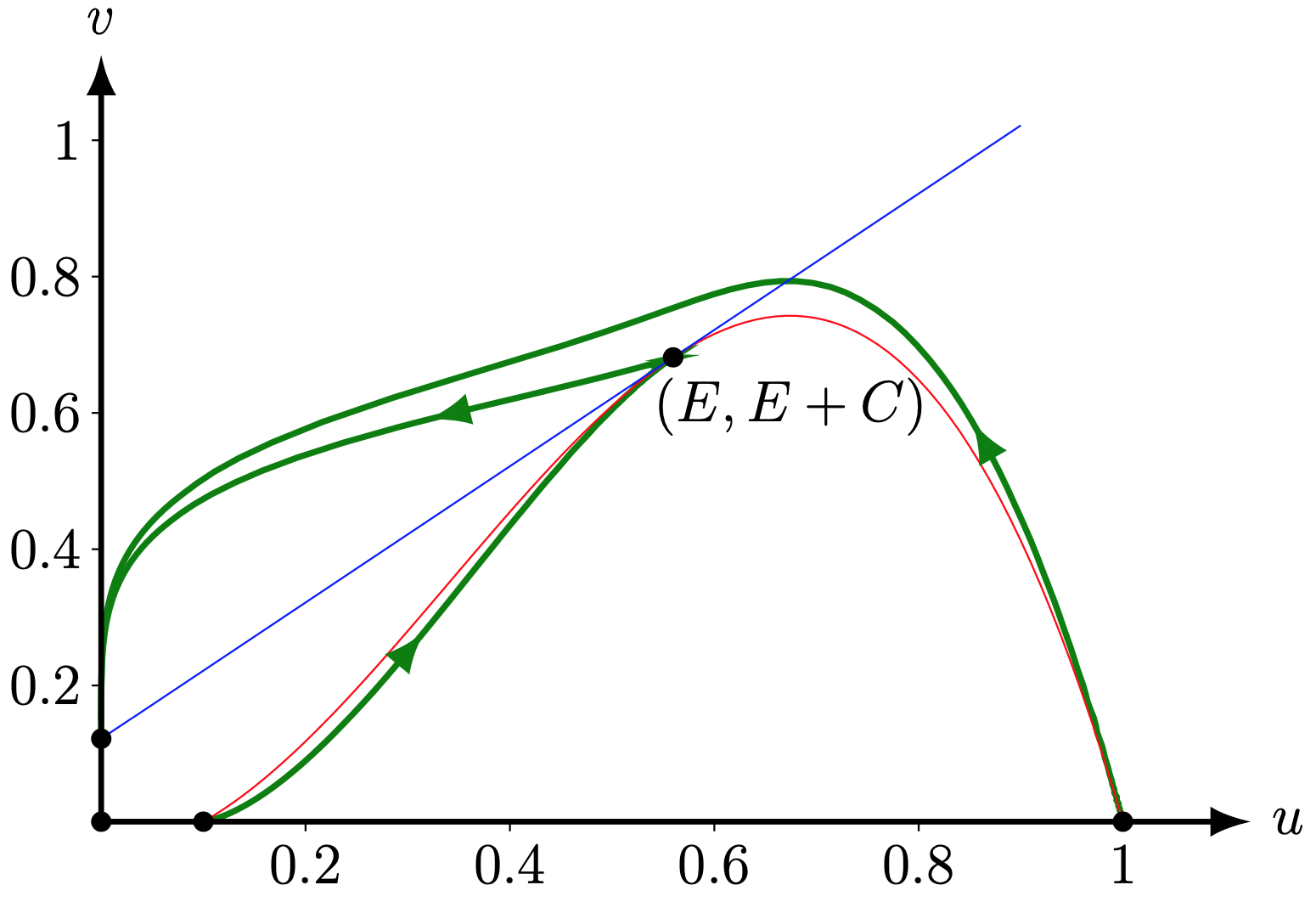}
\caption{If $M=0.1$, $A=0.08$, $Q=0.19$, $C=0.12176874$ and $S=0.08$, then the point $\left(0,C\right)$ is local attractor and the equilibrium $\left(E,E+C\right)$ is a cusp point.}\label{Fig4}
\end{figure}

\begin{theo}\label{SN}
Let the system parameters be such that $\Delta=0$~\eqref{eq6} and $M>0$, then system~\eqref{eq3} undergoes a saddle-node bifurcation at the equilibrium point $\left(E,E+C\right)$.
\end{theo}
\begin{proof}
I will proved that the system~\eqref{eq3} undergoes a saddle-node bifurcation at $Q=S\left(1+A+M+G\right)/\left(1-A+M+G\right)$ based on Sotomayor's theorem~\cite{perko}. If $\Delta=0$~\eqref{eq6}, then system~\eqref{eq3} has only one positive equilibrium point in the first quadrant. That is $\left(E,E+C\right)$, with $E=\dfrac{1}{2}\left(1-A+G+M\right)$.
	
The Jacobian matrix of the system~\eqref{eq3} evaluate at the equilibrium point $\left(E,E+C\right)$ is 
	\[\begin{aligned}
	J\left(E,E+C\right) & =\begin{pmatrix}
	QE\left(E+C\right) & -QE\left(E+C\right)  \\ 
	S\left(E+C\right)\left(E+A\right) & -S\left(E+C\right)\left(E+A\right) 
	\end{pmatrix}\\
	& =\left(A-G-M-1-2C\right)\begin{pmatrix}
	\dfrac{Q\left(A-G-M-1\right)}{4} & -\dfrac{Q\left(A-G-M-1\right)}{4} \\ 
	\dfrac{-S\left(A+G+M+1\right)}{4} & \dfrac{\left(A+G+M+1\right)}{4}
	\end{pmatrix}\\
	\end{aligned}.\]
	The vector form of system~\eqref{eq3} is given by
	\begin{equation} \label{mhtae12}
	f\left(u,v;Q\right) =\begin{pmatrix}
	\left(u+A\right)\left(1-u\right)\left(u-M\right)-Qv\\ 
	u-v+C
	\end{pmatrix}.
	\end{equation}
	Let $V=\begin{pmatrix}
	v_1 & v_2 \end{pmatrix}^T=\begin{pmatrix}
	1 & 1 \end{pmatrix}$ be the eigenvector corresponding to the eigenvalue $\Delta=0$ of $J\left(E,E+C\right)$. In addition, let $U=\begin{pmatrix}
	u_1 & u_2 \end{pmatrix}^T=\begin{pmatrix}
	-\dfrac{S\left(1+A+G+M\right)}{Q\left(1-A+G+M\right)} & 1 \end{pmatrix}$ be the eigenvector corresponding to the eigenvalue $\Delta=0$ of $\left(J\left(E,E+C\right)\right)^T$.
	
	On the other hand, differentiating the the vector function~\eqref{mhtae12} with respect to the bifurcation parameter $Q$ I obtain
	\[f_Q\left(u,v,Q\right)=\begin{pmatrix}
	\dfrac{A-1-G-M-2C}{2}\\ 
	0
	\end{pmatrix}.\]
	Therefore,
	\[Uf_Q\left(u,v;Q\right)=\dfrac{S\left(A+G+M+1\right)\left(1-A+2C+G+M\right)}{2Q\left(1-A+G+M\right)}\neq0.\]
	Next, I analyse the expression $U[D^2f\left(u,v;Q\right)\left(V,V\right)]$ where $V=\left(v_1,v_2\right)$ and $D^2f\left(u,v;Q\right)\left(V,V\right)$ is given by
	\[\begin{aligned}
	D^2f\left(u,v;Q\right)\left(V,V\right) & =\dfrac{\partial^2f\left(u,v;Q\right)}{\partial u^2}v_1v_1+\dfrac{\partial^2f\left(u,v;Q\right)}{\partial u\partial v}v_1v_2 +\dfrac{\partial^2f\left(u,v;Q\right)}{\partial v\partial u}v_2v_1+\dfrac{\partial^2f\left(u,v;Q\right)}{\partial v^2}v_2v_2\\
	D^2f\left(u,v;Q\right)\left(V,V\right)& = \begin{pmatrix}
	2\left(M-A-2\right)\\ 
	0
	\end{pmatrix}. 
	\end{aligned}\]
	Thus,
	\[\begin{aligned}
	U[D^2f\left(u,v;Q\right)]=\dfrac{2S\left(A+2-M\right)\left(A+G+M+1\right)}{Q\left(1-A+G+M\right)}\neq0
	\end{aligned}.\]
	Where $0<M<1$ and $A<1$. Therefore, by Sotomayor's theorem the system~\eqref{eq3} has a saddle-node bifurcation at $\left(E,E+C\right)$.
\end{proof}
\begin{figure}
\centering
\includegraphics[width=10cm]{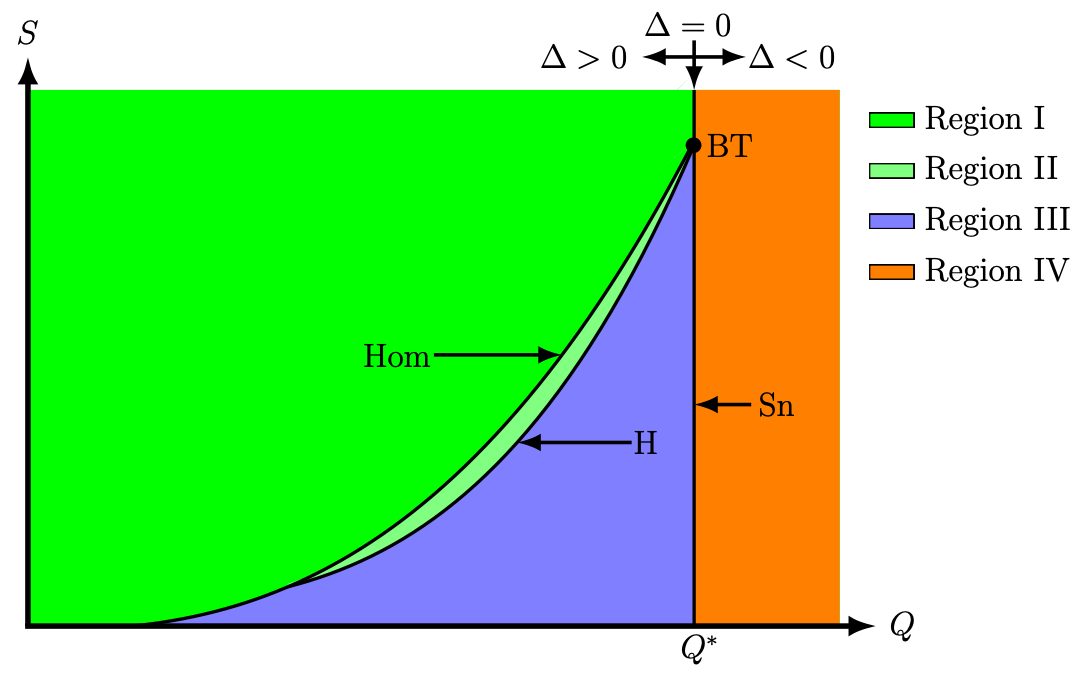}
\caption{The bifurcation diagram of system~\eqref{eq3} with strong Allee effect for $(M,A,C)=(0.1,0.08,0.19)$ fixed and created with the numerical bifurcation package MATCONT~\cite{matcont}. The curve H represents the Hopf curve Hom represents the homoclinic curve and Sn represent the saddle-node bifurcation. The point BT represent the Bogdanov--Takens bifurcation.}\label{Fig5}
\end{figure}
In order to get the bifurcation diagram I follow~\cite{arancibia} and we use the numerical bifurcation package MATCONT~\cite{matcont}. Furthermore, the bifurcation curves obtained from Theorems~\ref{BT},~\ref{SN} and~\ref{HOM} divide the $(Q,S)$ parameter space into four parts. When the parameters $Q,S$ are located in Region I (dark green area), the equilibrium point $P_{2}$ is stable, while in Region II (light green area) the equilibrium point is stable surrounded by an unstable limit cycle. Moreover, when the parameters $(Q,S)$ are located in Region III (light blue area), the equilibrium point $P_{2}$ is unstable. Additionally, we can observe that the modification of the parameter $S$ changes the stability of the positive equilibrium point $P_{2}$ of system~\eqref{eq3}, while the other equilibrium points $(0,0)$, $(1,0)$, $(M,0)$ and $(0,C)$ do not change their behaviour. Moreover, when parameters lie in the curve $Q=Q^*$ the equilibrium points $P_{1}$ and $P_{2}$ collapse, so that ~\eqref{eq3} has conditions for a saddle-node and Bogdanov--Takens bifurcation. Finally, when the parameters are located in Region IV, there are not positive equilibrium point in system~\eqref{eq3}, see Figure~\ref{Fig5}.

\section{Weak Allee effect ($M<0$)}\label{S3}
Next, I study the case of $M<0$, then the intersection of the cubic function $g\left(u\right)=\left(u+A\right)\left(1-u\right)\left(u-M\right)/Q$ and the straight line $h\left(u\right)=u+C$ depends on the value of the parameter $C$. I describe the different configurations for the solutions of equation~\eqref{eq4} given by $$u^3-\left(M+1-A\right)u^2-\left(A\left(M+1\right)-Q-M\right)u+AM+CQ=0,$$ and hence the number of positive equilibrium points (see Figure~\ref{Fig6}), below
\begin{enumerate}[label=(\roman*)]
\item \label{case1} If $C>-AM/Q$ and $M<0$, see top row of Figure~\ref{Fig6}.
\begin{enumerate}
\item \label{1.a} If $M+1-A>0$ or $M+1-A<0$ and $A\left(M+1\right)-Q-M<0$ or $M+1-A=0$ and $A\left(M+1\right)-Q-M>0$, then system~\eqref{eq3} has up to two positive equilibrium points in the first quadrant. 
\item \label{1.b} If $M+1-A<0$ and $A\left(M+1\right)-Q-M\geq0$ or $M+1-A=0$ and $A\left(M+1\right)-Q-M\leq0$, then system~\eqref{eq3} has no positive equilibrium points in the first quadrant. 
\end{enumerate}
\item \label{case2} If $C=-AM/Q$ and $M<0$, see middle row of Figure~\ref{Fig6}.
\begin{enumerate}
\item \label{2.a} If $M+1-A>0$ and $A\left(M+1\right)-Q-M<0$, then system~\eqref{eq3} has up to two positive equilibrium points in the first quadrant.
\item \label{2.b} If $M+1-A>0$ and $A\left(M+1\right)-Q-M\geq0$ or $M+1-A<0$ and $A\left(M+1\right)-Q-M>0$ or $M+1-A=0$ and $A\left(M+1\right)-Q-M>0$, then system~\eqref{eq3} has one positive equilibrium point in the first quadrant.  
\item \label{2.c} If $M+1-A<0$ and $A\left(M+1\right)-Q-M\leq0$ or $M+1-A=0$ and $A\left(M+1\right)-Q-M\leq0$, then system~\eqref{eq3} has no positive equilibrium points in the first quadrant.
\end{enumerate}
\item \label{case3} If $C<-AM/Q$ and $M<0$, see bottom row of Figure~\ref{Fig6}.
\begin{enumerate}
\item \label{3.a} If $M+1-A\geq0$ and $A\left(M+1\right)-Q-M\geq0$ or $M+1-A\leq0$ and $A\left(M+1\right)-Q-M<0$ or $M+1-A<0$ and $A\left(M+1\right)-Q-M>0$, then system~\eqref{eq3} has one positive equilibrium point in the first quadrant.
\item \label{3.b} If $M+1-A>0$ and $A\left(M+1\right)-Q-M<0$, then system~\eqref{eq3} has up to three positive equilibrium points in the first quadrant.
\end{enumerate}
\end{enumerate}

\begin{figure}
\centering
\includegraphics[width=15cm]{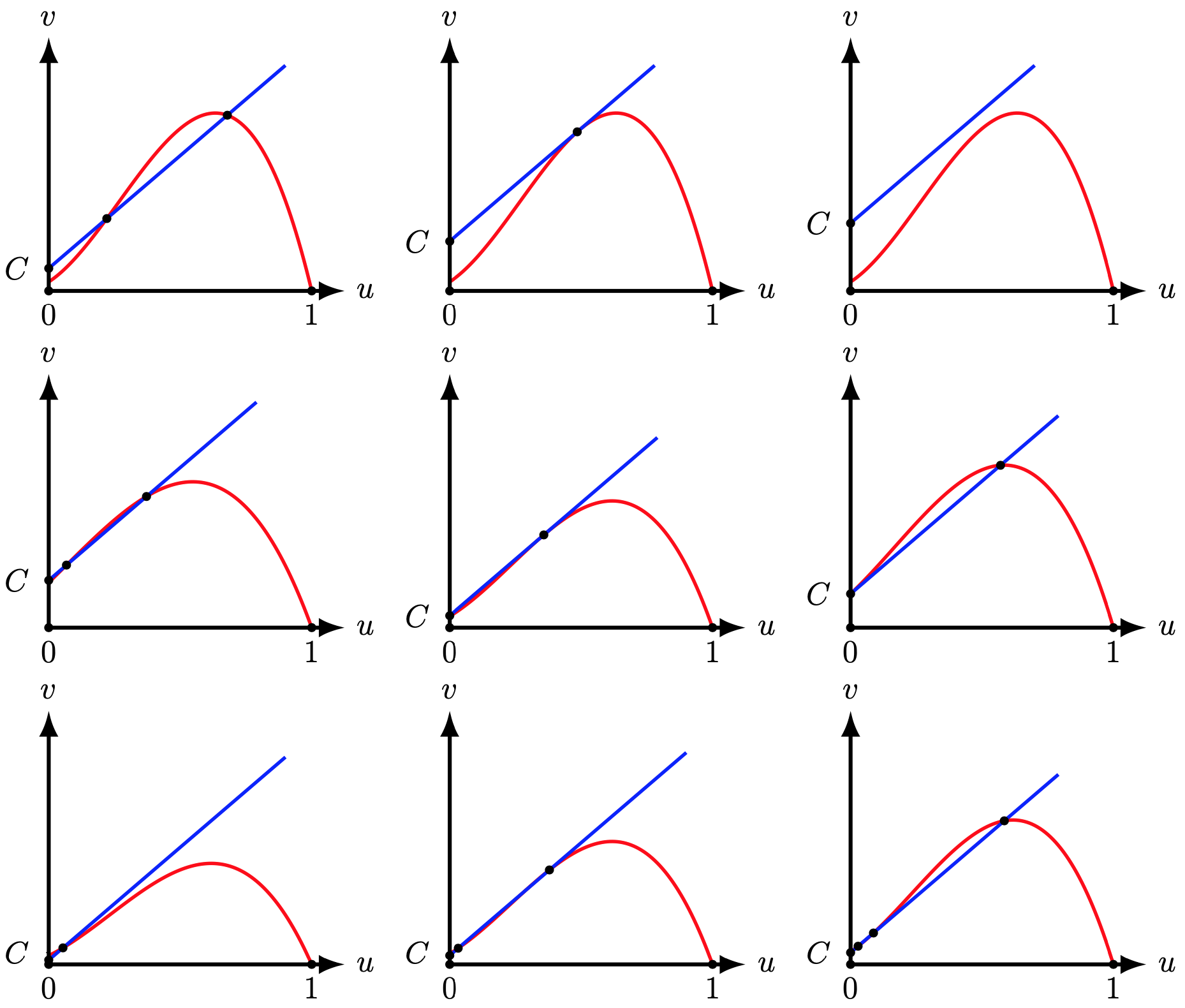}
\caption{Intersection of the function $g\left(u\right)=\left(u+A\right)\left(1-u\right)\left(u-M\right)/Q$ (red) and the straight line $h\left(u\right)=u+C$ (blue) for system~\eqref{eq3} affected by weak Allee effect, i.e $M<0$. In the top row I consider the case of $C>-AM/Q$~\eqref{case1}, while in the middle row the case of $C=-AM/Q$~\eqref{case2} and in the bottom row the case of $C<-AM/Q$~\eqref{case3}.} \label{Fig6}
\end{figure}

Next, I study cases~\ref{3.a} and~\ref{3.b} in which equation~\eqref{eq4} can always has one positive root namely $W$. I divide again the cubic equation~\eqref{eq4} by $\left(u-W\right)$, I obtain the second order polynomial
\begin{equation}\label{mhtwwae7}
u^2+u\left(A+W-M-1\right)+ \left(M+Q-A\left(M+1\right)+W\left(A+W-M-1\right)\right)=0,
\end{equation}
As a result, I get that $Q=\left(W-1\right)\left(W-M\right)\left(A+W\right)/\left(W+C\right)$ and thus, $0<W<1$ and $M<0$. Therefore, if $AM-CQ<0$, $M+1-A>0$ and $A\left(M+1\right)-Q-M<0$ (see case~~\ref{3.b}), then the solutions of equation~\eqref{mhtwwae7} are
\[u_{1}=\dfrac{1}{2}\left(W-A+M+1-\sqrt{\Delta}\right) \quad \text{and}\quad u_{2}=\dfrac{1}{2}\left(W-A+M+1+\sqrt{\Delta}\right).\] 
with $\Delta=\left(W-A+M+1\right)^2-4\left(M+Q-A\left(M+1\right)+W\left(A+W-M-1\right)\right)$.
Therefore, system~\eqref{eq4} has three positive equilibrium points $\left(W,W+C\right)$, $P_1=\left(u_1,u_1+C\right)$ and $P_2=\left(u_2,u_2+C\right)$.  

\subsection{Nature of equilibrium points}

To determine the nature of the equilibrium points of system~\eqref{eq3} with weak Allee effect, i.e $M<0$, I consider the Jacobian matrix~\eqref{eq7} which is given by
\begin{equation*}
J\left(u,v\right)=\begin{pmatrix}
-5u^4+4\left(M-C-A+1\right)u^3+\beta  & -Qu\left(u+C\right) \\ 
Sv\left(A+C+2u-v\right)  &  S\left(C+u-2v\right)\left(A+u\right) 
\end{pmatrix}.
\end{equation*}
With $\beta=3\left(A+C-M-AC+AM+CM\right)u^2+2\left(AC-AM-CM-Qv+ACM\right)u-C\left(AM+Qv\right)$.
\begin{lemm}
The equilibrium point $\left(0,0\right)$ is a repeller point and $\left(1,0\right)$ is a saddle point.
\end{lemm}
\begin{proof}
The Jacobian matrix~\eqref{eq7} evaluate at the equilibrium point $\left(0,0\right)$ gives
\[ J\left(0,0\right)=\begin{pmatrix} -ACM  & 0 \\ 0  &  ACS \end{pmatrix}.\]
Hence, $\det\left(J\left(0,0\right)\right)=-A^2C^2MS>0$ and $\tr\left(J\left(0,0\right)\right)=-ACM+ACS>0$, since $M<0$. Therefore, the equilibrium $\left(0,0\right)$ is a repeller point.
Similarly, the Jacobian matrix~\eqref{eq7} evaluate at the equilibrium point $\left(1,0\right)$ gives
\[ J\left(1,0\right)=\begin{pmatrix} -\left(1-M\right)\left(C+1\right)\left(A+1\right)  & -Q\left(C+1\right) \\  0  &  S\left(C+1\right)\left(A+1\right) \end{pmatrix}.\]
Hence, $\det\left(J\left(1,0\right)\right)=-S\left(1-M\right)\left(C+1\right)^2\left(A+1\right)^2<0$ since $0<M<1$. Therefore, the equilibrium $\left(1,0\right)$ is a saddle point.
\end{proof}
\begin{theo}\label{0C}
The equilibrium point $\left(0,C\right)$ is
\begin{enumerate}
\item a saddle point if $C<-\dfrac{AM}{Q}$,
\item a saddle-node if $C=-\dfrac{AM}{Q}$, and
\item an attractor point if $C>-\dfrac{AM}{Q}$.
\end{enumerate}
\end{theo}
\begin{proof}
The Jacobian matrix~\eqref{eq7} evaluate at the equilibrium point $\left(0,C\right)$ gives
\[ J\left(0,C\right)=\begin{pmatrix} -C\left(AM+CQ\right)  & 0 \\ ACS  &  -ACS \end{pmatrix}.\]
Hence, $\det\left(J\left(0,C\right)\right)=AC^2S\left(AM+CQ\right)$ and $\tr\left(J\left(0,C\right)\right)=-C\left(AS+AM+CQ\right)$. Therefore, if $C<-AM/Q$, then $\det\left(J\left(0,C\right)\right)<0$ and thus the equilibrium point $\left(0,C\right)$ is a saddle point. Moreover, if $C>-AM/Q$, then $\det\left(J\left(0,C\right)\right)>0$ and $\tr\left(J\left(0,C\right)\right)=-C\left(AS+AM+CQ\right)>0$, since $M<0$ and $C>-AM/Q$. Therefore, the equilibrium point $\left(0,C\right)$ is an attractor point. Finally, if $C=-AM/Q$, then the equilibrium points $\left(0,C\right)$ and $P_1$ collapse, see Theorem~\ref{0C}. The center manifold theorem~\cite{perko} will be used to prove the stability of the singularity $\left(0,C\right)$ when $C=-AM/Q$.
	
Setting $\left(u,v\right)\rightarrow\left(X,Y+C\right)$, I move the equilibrium point $\left(0,C\right)$ of system~\eqref{eq3} to the origin. Therefore, I obtain the equivalent system
\begin{equation}\label{mhtwwaec1}
\begin{aligned}
\dfrac{dX}{dt} &=X\left( \left(1-X\right)\left(X + A\right)\left(X-M\right)-Q\left(Y+C\right)\right)\left(X+C\right),\\
\dfrac{dY}{dt} &=S\left(X-Y\right)\left(X+A\right)\left(Y+C\right).
\end{aligned}
\end{equation}
The diagonal form of a two--dimensional system can be written by
\begin{equation*}
\begin{aligned}
\dfrac{dx}{dt} &=\delta x+\Phi\left(x,y\right),\\
\dfrac{dy}{dt} &=\epsilon y+\Psi\left(x,y\right).
\end{aligned}
\end{equation*}
In addition, the flow on the center manifold is defined by the system of differential equation
\begin{equation}\label{mhtwwaec3}
\dfrac{dx}{dt}=\varepsilon x+\Phi\left(x,h\left(x\right)\right)
\end{equation}
Thus, system~\eqref{mhtwwaec1} can be  written by
\[\begin{aligned}
\dfrac{dX}{d\tau} =&-C\left(AM+CQ\right)X +\left(AC-AM-CM-CQ-QY+ACM\right)X^2\\
&+\left(A+C-M-AC+AM+CM\right) X^3 +\left(M-C-A+1\right)X^4-X^5-CQXY,\\
\dfrac{dY}{d\tau} =&-ACSY+CSX^2-ASY^2-SXY^2+SX^2Y+ACSX+ ASXY-CSXY.
\end{aligned}\]
So, we have that 
\[\begin{aligned}
\delta =&-C\left(AM+CQ\right), \\
\epsilon =&-ACS, \\
\Phi\left(X,Y\right) =&+\left(AC-AM-CM-CQ-QY+ACM\right)X^2+\left(A+C-M-AC+AM+CM\right) X^3\\
& +\left(M-C-A+1\right)X^4-X^5-CQXY\\
\Psi\left(X,Y\right) =&CSX^2-ASY^2-SXY^2+SX^2Y+ACSX+ ASXY-CSXY.
\end{aligned}\]
Considering the function $h\left(X\right)$ as the local center manifold defined by 
\begin{equation}\label{mhtwwaec4}
h\left(X\right)=aX^2+bX^3+cX^4+0\left(X^5\right)
\end{equation}
and 
\begin{equation}\label{mhtwwaec5}
Dh\left(X\right)=2aX+3bX^2+4cX^3+0\left(X^4\right).
\end{equation}
In addition, the function $h\left(X\right)$ satisfies 
\begin{equation}\label{mhtwwaec6}
Dh\left(X\right)\left(\delta X+\Phi\left(X,h_5\left(X\right)\right)\right)-\left(\epsilon h\left(X\right)+\Psi\left(X,h\left(X\right)\right)\right)=0
\end{equation}
Thus, replacing~\eqref{mhtwwaec4} and~\eqref{mhtwwaec5} into equation~\eqref{mhtwwaec6} and setting the coefficients $a$, $b$, and $c$ solving the equation~\eqref{mhtwwaec6}, we have that 
\[\begin{aligned}
a=& \dfrac{1}{A},\\
b=&-\dfrac{2AC-2CM-AS+CS+2ACM}{A^2CS},\\
c=&\dfrac{6C^2M^2\left(A^2-2A+1\right)+CM\zeta+A^2\left(2C^2S+6C^2-7CS+S^2\right)+CS\left(3AC-2AS+CS\right)}{A^3C^2S^2}.
\end{aligned}\]
With $\zeta=12A^2C-5CS-12AC+3ACS-7A^2S+5AS$. Therefore, 
\[\begin{aligned}
h\left(X\right)=&\dfrac{1}{A}X^2+\dfrac{2AC-2CM-AS+CS+2ACM}{A^2CS}X^3\\
&+\dfrac{6C^2M^2\left(A^2-2A+1\right)+CM\zeta_8+A^2\eta_8+CS\left(3AC-2AS+CS\right)}{A^3C^2S^2}X^4+0\left(X^5\right).
\end{aligned}\]
Thus, replacing $h\left(X\right)$ in equation~\eqref{mhtwwaec3}; we have that the flow on the centre manifold is
\begin{equation*}\label{mhtwwaec7}
\dfrac{dX}{d\tau}=\dfrac{1}{A^3C^2S^2}\left(\vartheta X^2+\iota X^3-\kappa X^4+\nu X^5+\xi X^6+\mathcal{O}\left(X^7\right)\right).
\end{equation*}
	With\\
	\[\begin{aligned}
	\vartheta=& A^2C^4S^2\left(1-A\right)\\
	\iota=& A^2C^3S^2\left(2-A-C\right)\\
	\kappa=&AC^2S\left(2AC-2CM-AS+CS+2ACM+ACS-AMS\right)\\
	\nu=&C\left(6A^2C^2+6C^2M^2+A^2S^2+C^2S^2+6A^2C^2M^2-12AC^2M-2ACS^2+3AC^2S-7A^2CS-5C^2MS\right.\\
	&-12AC^2M^2+12A^2C^2M+2A^2C^2S+A^2MS^2-ACMS^2+2ACM^2S+3AC^2MS-9A^2CMS\\
	&\left.-2A^2CM^2S+5ACMS\right)\\
	\xi=&M\left(6A^2C^2M^2+12A^2C^2M+2A^2C^2S+6A^2C^2-7A^2CMS-7A^2CS+A^2S^2-12AC^2M^2+3AC^2MS\right.\\
	&\left.-12AC^2M+3AC^2S+5ACMS-2ACS^2+6C^2M^2-5C^2MS+C^2S^2\right)
	\end{aligned}\]
Considering that series expansion of the function $h\left(X\right)$, it also is approximate the shape of the local center manifold. Therefore, we have that the point $\left(0,C\right)$ is saddle-node.
\end{proof}

First, I recall that the determinant~\eqref{eq9} and the trace~\eqref{eq10} of the Jacobian matrix~\eqref{eq8} are given by
\[\begin{aligned}
\det\left(J\left(u,u+C\right)\right) =& Su\left(A+u\right)\left(C+u\right)^2\left(Q-J_{11}\left(u\right)\right)	\,,\\
\tr\left(J\left(u,u+C\right)\right) 	=& \left(C+u\right)\left(uJ_{11}-S\left(A+u\right)\right)\,.
\end{aligned}\]
With $J_{11}\left(u\right)=A-M+2u+AM-2Au+2Mu-3u^2$. Note that the signs of the determinant~\eqref{eq9} depends on the value of $Q-J_{11}\left(u\right)$ and the signs of the trace~\eqref{eq10} depends on the value of $uJ_{11}\left(u\right)-S\left(A+u\right)$. 

Next, I discuss the stability of the equilibrium points $P_1$ and $P_2$ of cases~\ref{1.a} and~\ref{2.a}. The stability of these points are the same as the stability of the equilibrium point showed in Theorems~\ref{p1} and~\ref{p2}. That is, Theorems~\ref{p1} and~\ref{p2} also holds for weak Allee effect ($M<1$). Moreover, in case~\ref{2.b} the equilibrium point $P_1$ crosses to the second or third quadrant and thus the only positive equilibrium point is $P_2$. In this case the stability of the equilibrium point $P_2$ is the same as the stability showed in Theorem~\ref{p2}. In cases~\ref{1.b} and~\ref{2.c} system\eqref{eq3} has no positive equilibrium points in the first quadrant. Therefore, the equilibrium point $\left(0,C\right)$ is global attractor.

On the other hand, if $M+1-A\geq0$ and $A\left(M+1\right)-Q-M\geq0$ or $M+1-A\leq0$ and $A\left(M+1\right)-Q-M<0$ or $M+1-A<0$ and $A\left(M+1\right)-Q-M>0$, then system~\eqref{eq3} has only one positive equilibrium point in the first quadrant, which I denote by $\left(W,W+C\right)$ where $0<W<1$.   

\begin{lemm}
Let the system parameters of system~\eqref{eq3} be such that $M<0$, $0<C<-AM/Q$ and the conditions of case~\ref{3.a} are met. Then system~\eqref{eq3} has only one positive equilibrium point $(W,W+C)$ which is a stable node.
\end{lemm}
\begin{proof}
Evaluating $Q-J_{11}\left(u\right)$~\eqref{eq9} at $W$ gives
\[Q-J_{11}\left(W\right)=-(A(1+M)-Q-M)-2W(M+1-A)+3W^2.\]
Hence, it is clear that if $A(1+M)-Q-M\leq0$ and $M+1-A\leq0$, then $Q-J_{11}\left(W\right)>0$ and thus $\det\left(J\left(W,W+C\right)\right)>0$. Moreover, rewriting equation~\eqref{eq4} as $W=W(M+1-A)+(A(M+1)-Q-M)-(AM+CQ)/W$, then $Q-J_{11}\left(W\right)$ become
\[Q-J_{11}\left(W\right)=2(A(1+M)-M-Q)+W(M+1-A)-\dfrac{3(AM+CQ)}{W}.\]
Hence, it is also clear that if $A(1+M)-Q-M\geq0$ and $M+1-A\geq0$, then $Q-J_{11}\left(W\right)>0$ and thus $\det\left(J\left(W,W+C\right)\right)>0$. Similarly,  I can also rewriting equation~\eqref{eq4} as $A(M+1)-Q-M=W^2-(M+1-A)W+(AM+CQ)/W$, then $Q-J_{11}\left(W\right)$ now become
\[Q-J_{11}\left(W\right)=-W(M+1-A)+2W^2-\dfrac{AM+CQ}{W}.\]
Hence, it is clear again that if $A(1+M)-Q-M>0$ and $M+1-A<0$, then $Q-J_{11}\left(W\right)>0$ and thus $\det\left(J\left(W,W+C\right)\right)>0$.
Then, the behaviour of the equilibrium point $(W,W+C)$ depends on the trace~\eqref{eq10} of the Jacobian matrix~\eqref{eq8} at the equilibrium point $(W,W+C)$. Evaluating $uJ_{11}\left(u\right)-S\left(A+u\right)$~\eqref{eq10} at $W$ gives
\[uJ_{11}\left(W\right)-S\left(A+W\right)=-W^3-W(A(1+M)-M-2Q)+2(AM+CQ)-S(A+W)<0.\]
Since $A(1+M)-M-2Q>0$ and $AM+CQ<0$. Therefore, the equilibrium point is always a stable node, see Figure~\ref{Fig10}.
\end{proof}

\begin{figure}
\centering
\includegraphics[width=8cm]{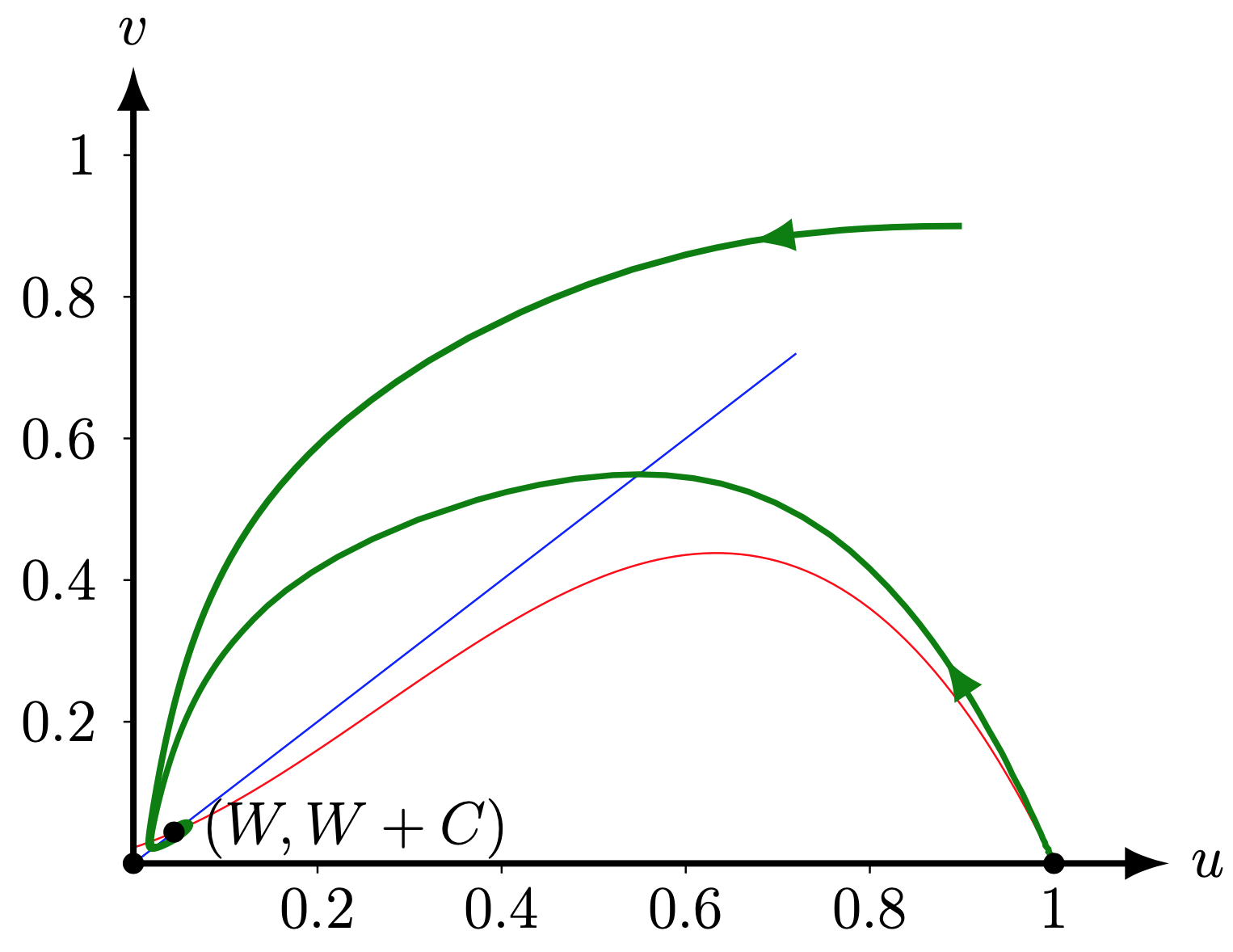}
\includegraphics[width=8cm]{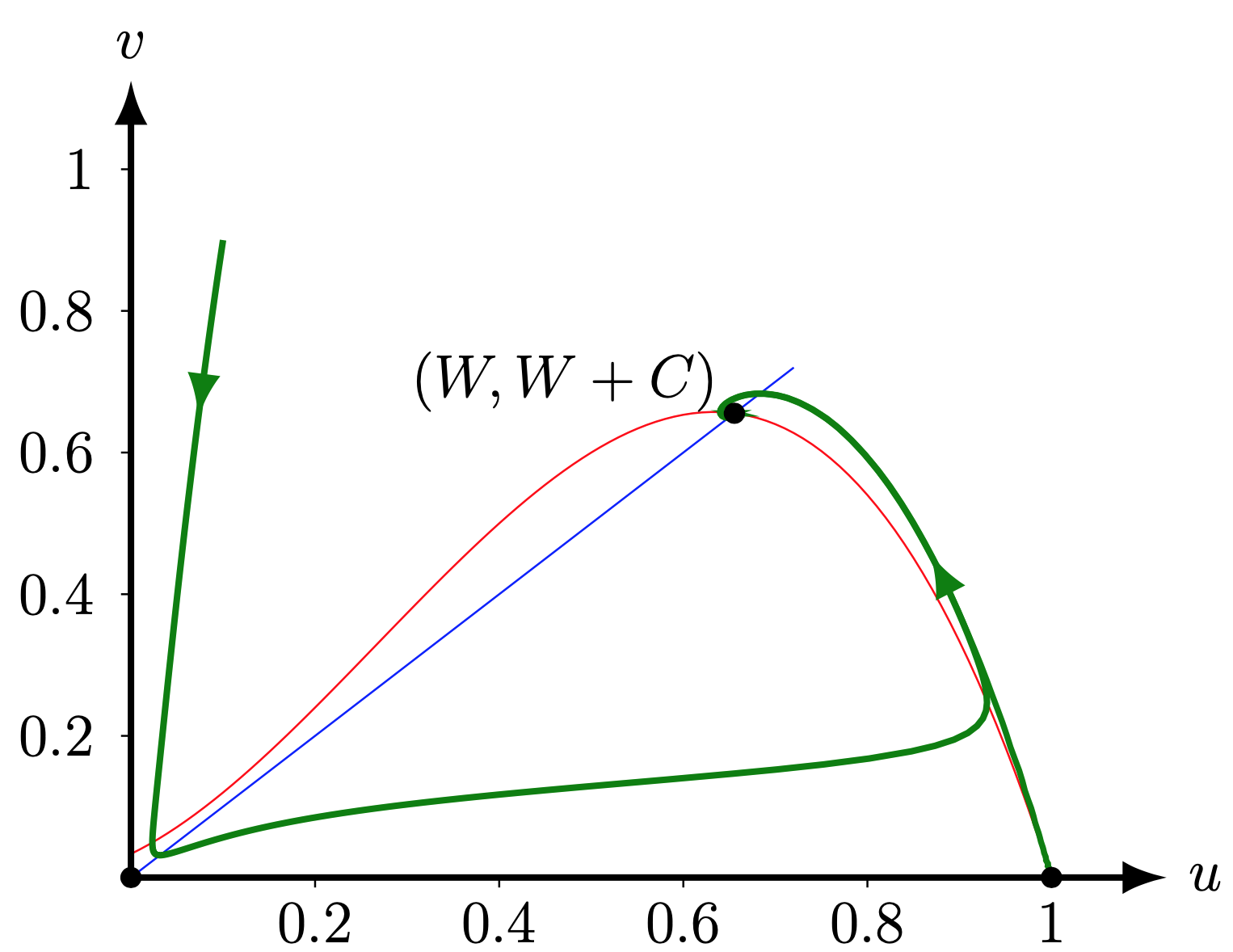}
\caption{The blue (red) curve represents the predator (prey) nullcline. If $A=0.4$; $Q=0.53$; $M=-0.1$ and $C=0.06$ are fixed, then system~\eqref{eq3} has one positive equilibrium point $(W,W+C)$. Moreover, in the left panel ($S=0.15$) and in the right panel ($S=0.25$) the equilibrium point $(W,W+C)$ is a stable node.}\label{Fig10}
\end{figure}

Next, I study case~\ref{3.b} ($M<0$, $C<-AM/Q$, $M+1-A>0$ and $A(M+1)-Q-M<0$) in which system~\eqref{eq3} has three equilibrium point in the first quadrant namely $(W,W+C)$, $P_1=(u_1,u_1+C)$ and $P_2=(u_2,u_2+C)$ with $u_{1,2}$ defined in~\eqref{eq6}. 

\begin{theo}
Let the system parameters of system~\eqref{eq3} be such that $M<0$, $C<-AM/Q$, $M+1-A>0$ and $A(M+1)-Q-M<0$ (see case~\ref{3.a}). Then, $(W,W+C)$ is\\
\begin{enumerate} 
\item a stable node if $W^3-(Q+M-A(M+1))W-2(AM+CQ)<0$ and \\
 $S>\dfrac{3(AM+CQ)+W((A-M-1)W+2(Q+M-A(1+M))-Q)}{W+A}$;
\item an unstable node if $W^3-(Q+M-A(M+1))W-2(AM+CQ)<0$ and \\
$S<\dfrac{3(AM+CQ)+W((A-M-1)W+2(Q+M-A(1+M))-Q)}{W+A}$;
\item a center if $W^3-(Q+M-A(M+1))W-2(AM+CQ)<0$ and \\
$S=\dfrac{3(AM+CQ)+W((A-M-1)W+2(Q+M-A(1+M))-Q)}{W+A}$;
\item a saddle point if $W^3-(Q+M-A(M+1))W-2(AM+CQ)>0$.
\end{enumerate} 
\end{theo}
\begin{proof}
Evaluating $Q-J_{11}\left(u\right)$~\eqref{eq9} at $W$ gives $Q-J_{11}\left(u\right)=W^3+(A(M+1)-Q-M)W-2(AM+CQ)$. Hence, the behaviour of the equilibrium point $(W,W+C)$ depends on the behaviour of the equilibrium point depends on the value of  $Q-J_{11}\left(W\right)$. If $Q-J_{11}\left(W\right)<0$, then $\det\left(J\left(W,W+C\right)\right)<0$ and thus the equilibrium point $(W,W+C)$ is a saddle point. While, if $Q-J_{11}\left(W\right)>0$, then $\det\left(J\left(W,W+C\right)\right)>0$ and thus the behaviour of the equilibrium point $(W,W+C)$ depends on the trace~\eqref{eq10} of the Jacobian matrix~\eqref{eq8} at the equilibrium point $(W,W+C)$. Which is given by $$\tr(J(W,W+C))=(C+W)(3(AM+CQ)-W((M+1-A)W-2(A(M+1)-Q-M+Q)-S(A+W)).$$ Hence, there are parameter values such that $\tr(J(W,W+C))>0$, $\tr(J(W,W+C))=0$ or $\tr(J(W,W+C))<0$.
\end{proof}

\begin{theo}
Let the system parameters of system~\eqref{eq3} be such that $M<0$, $C<-AM/Q$, $M+1-A>0$ and $A(M+1)-Q-M<0$ (see case~\ref{3.a}). Then, $P_1$ is
\begin{enumerate} 
\item a stable node if $W-u_1>0$ and $S>\dfrac{4u_{1}( Q+(u_{1}-W)\sqrt{\Delta})}{A+u_{1}}$,
\item an unstable node if $W-u_1>0$ and $S<\dfrac{4u_{1}( Q+(u_{1}-W)\sqrt{\Delta})}{A+u_{1}}$,
\item a centre if $W-u_1>0$ and $S=\dfrac{4u_{1}( Q+(u_{1}-W)\sqrt{\Delta})}{A+u_{1}}$,
\item a saddle if $W-u_1<0$.
\end{enumerate} 
\end{theo}
\begin{proof}
Evaluating $Q-J_{11}\left(u\right)$~\eqref{eq9} at $u_1$ gives $Q-J_{11}\left(u\right)=(W-u_{1})$. Hence, if $W<u_{1}$ then $\det\left(J\left(P_1\right)\right)<0$ and thus the equilibrium point $P_1$ is a saddle point. Moreover, if $W=u_{1}$ then $\det\left(J\left(P_1\right)\right)=0$ and thus the equilibrium point $P_1$ and $\left(W,W+C\right)$ collapse. While, if $W>u_{1}$ then $\det\left(J\left(P_1\right)\right)>0$ and thus the stability of the equilibrium point $P_1$ depends on the trace~\eqref{eq10} of the Jacobian matrix~\eqref{eq8} at the equilibrium point $P_1$. Evaluating $uJ_{11}\left(u\right)-S\left(A+u\right)$~\eqref{eq10} at $u_1$ gives
\[uJ_{11}\left(u\right)-S\left(A+u\right)=\dfrac{1}{2}(4u_{1}( Q+(u_{1}-W)\sqrt{\Delta})-S(A+u_{1})).\]
Therefore, there are parameter values such that $\tr(J(P_1))>0$, $\tr(J(P_1))=0$ or $\tr(J(P_1))<0$.
\end{proof}
\begin{theo}
Let the system parameters of system~\eqref{eq3} be such that $M<0$, $C<-AM/Q$, $M+1-A>0$ and $A(M+1)-Q-M<0$ (see case~\ref{3.a}). Then, $P_2$ is
\begin{enumerate} 
\item a stable node if $u_2-W>0$ and $S>\dfrac{4u_{2}( Q-(u_{2}-W)\sqrt{\Delta})}{A+u_{2}}$,
\item an unstable node if $u_2-W>0$ and $S<\dfrac{4u_{2}( Q-(u_{2}-W)\sqrt{\Delta})}{A+u_{2}}$,
\item a centre if $u_2-W>0$ and $S=\dfrac{4u_{2}( Q-(u_{2}-W)\sqrt{\Delta})}{A+u_{2}}$,
\item a saddle if $u_2-W<0$.
\end{enumerate} 
\end{theo}
\begin{proof}
Evaluating $Q-J_{11}\left(u\right)$~\eqref{eq9} at $u_2$ gives $Q-J_{11}\left(u\right)=(u_{2}-W)$. Hence, if $W>u_{2}$ then $\det\left(J\left(P_2\right)\right)<0$ and thus the equilibrium point $P_2$ is a saddle point. Moreover, if $W=u_{2}$ then $\det\left(J\left(P_2\right)\right)=0$ and thus the equilibrium point $P_2$ and $\left(W,W+C\right)$ collapse. While, if $W<u_{2}$ then $\det\left(J\left(P_2\right)\right)>0$ and thus the stability of the equilibrium point $P_2$ depends on the trace~\eqref{eq10} of the Jacobian matrix~\eqref{eq8} at the equilibrium point $P_2$. Evaluating $uJ_{11}\left(u\right)-S\left(A+u\right)$~\eqref{eq10} at $u_2$ gives
\[uJ_{11}\left(u\right)-S\left(A+u\right)=\dfrac{1}{2}(4u_{2}( Q+(u_{2}-W)\sqrt{\Delta})-S(A+u_{2})).\]
Therefore, there are parameter values such that $\tr(J(P_2))>0$, $\tr(J(P_2))=0$ or $\tr(J(P_2))<0$.
\end{proof}

\begin{theo}\label{L1}
Let the system parameters of system~\eqref{eq3} be such that $M<0$, $C<-AM/Q$, $M+1-A>0$, $A(M+1)-Q-M<0$ and $\Delta=0$. Then, the equilibrium point $P_1$ and $P_2$ collapse and thus $(W,W+C)<P_{1}=P_{2}=(L_{1},L_{1}+C)$ with $L_1=(1-A+W+M)/2$. Moreover, the equilibrium point $(L_{1},L_{1}+C)$ is:
\begin{enumerate} 
\item a saddle-node attractor if $Q>\dfrac{S(M-W+A+1)}{2(M-W-A+1)}$,
\item a saddle-node repeller if $Q<\dfrac{S(M-W+A+1)}{2(M-W-A+1)}$.
\end{enumerate} 
\end{theo}
\begin{proof}
If $\Delta=0$, then the equilibrium points $P_{1}$ and $P_{2}$ collapse and thus $\det\left(J\left(L_1\right)\right)=0$ since $\Delta=0$. Therefore, the stability of the equilibrium point $(L_1,L_1+C)$ depends on the trace~\eqref{eq10} of the Jacobian matrix~\eqref{eq8} at the equilibrium point $L_1$ which is given by
	\[\tr(J(L_{1},L_{1}+C)) = \dfrac{1}{4}(2C+M-W-A+1-P)(2Q(M-W-A+1-P) -S(M-W+A+1-P)).\]
Therefore, the behaviour of the equilibrium point $(L_{1},L_{1}+C)$ depends on the value of $2Q(M-W-A+1-P)-S(M-W+A+1-P)$, see Figure~\ref{Fig11}.
\end{proof}

\begin{figure}
\centering
\includegraphics[width=10cm]{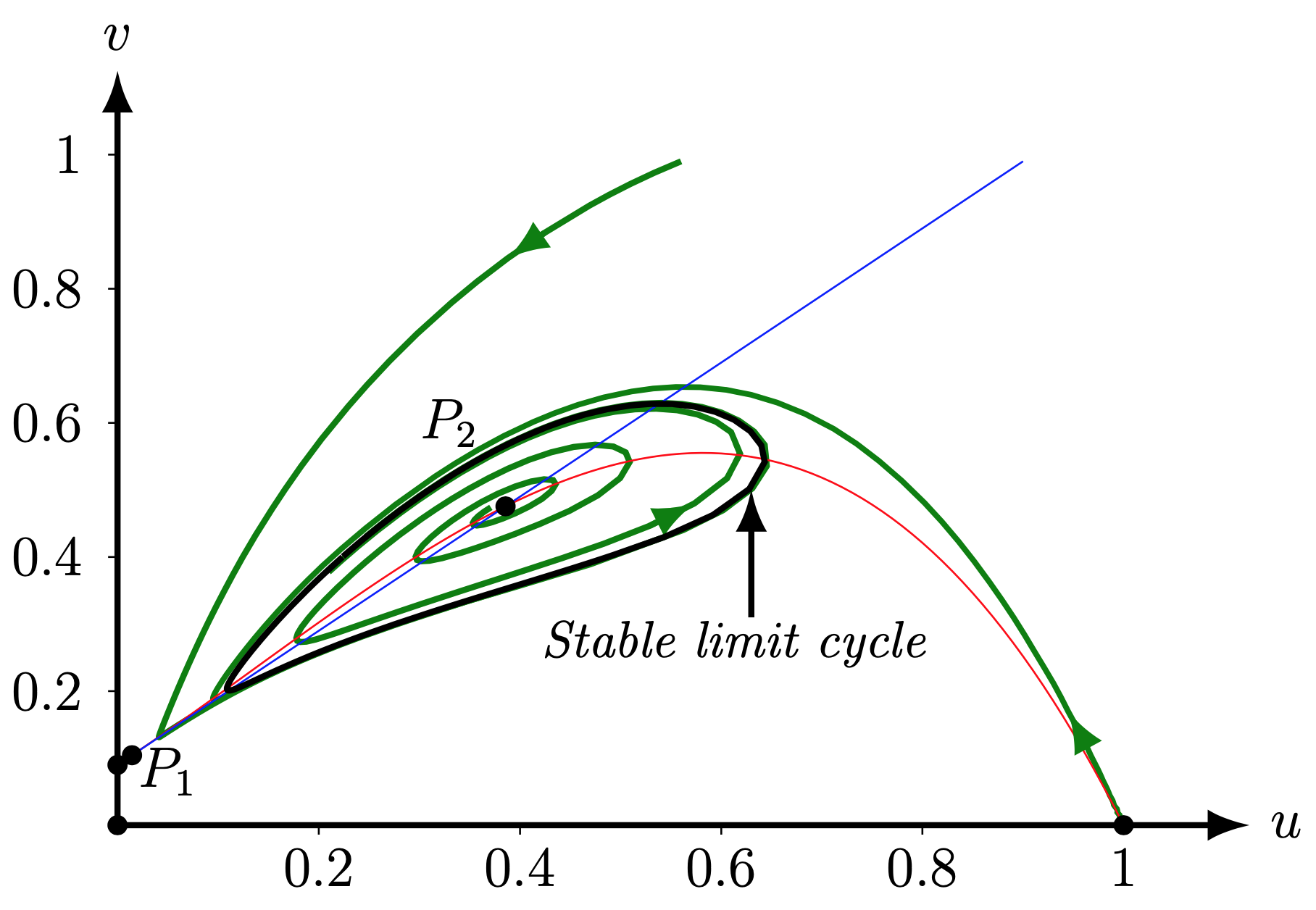}
\caption{If $A=0.5$; $Q=0.5555556$; $M=-0.1$, $S=0.15$ $C=0.09$, then the equilibrium $(1,0)$, $(0,0)$ and $P_{1}$ are saddle points, $(0,C)$ is a saddle-node and $P_{2}$ is unstable node surrounded by a stable limit cycles.}\label{Fig11}
\end{figure}

Note that there are conditions in the system parameter for which the equilibrium points $(W,W+C)$ and $P_{1}$ can collapse and thus $(W,W+C)=P_{1}=L_2<P_{2}$. Moreover, by continuity the equilibrium points $(W,W+C)$ and $P_2$ can also collapse and thus $P_{1}<(W,W+C)=P_{2}=L_3$. Finally, the equilibrium points $P_{1}$ and $P_{2}$ can collapse again and thus $P_{1}=P_{2}=L_4<(W,W+C)$. The stability of the equilibrium points $(L_2,L_2+C)$, $(L_3,L_3+C)$ and $(L_4,L_4+C)$ can be proved by following Theorem~\ref{L1}.

\subsection{Bifurcation Analysis}
In this section I will discuss the bifurcation analysis of system~\eqref{eq3} for $\Delta=0$ and $M<0$. Additionally, if $\Delta=0$ then the equilibrium points $P_{1}$ and $P_{2}$ collapse and thus I have the following cases: $(W,W+C)<P_{1}=P_{2}$ and $P_{1}=P_{2}<(W,W+C)$.

\begin{theo}\label{BT2}
Let the system parameters be such that $\Delta=0$~\eqref{eq6}, $M<0$ and $Q=S(A+W+M+1)/(W-A+M+1)$, then system~\eqref{eq3} undergoes a Bogdanov--Takens bifurcation at the equilibrium point $(L_1,L_1+C)$ (similarly for $(L_4,L_4+C)$).
\end{theo}
\begin{proof}
	If $Q=S(A+W+M+1)/(W-A+M+1)$, then the trace is $\tr\left(J(L_1,L_1+C)\right)=0$ and the Jacobian matrix~\eqref{eq8} at the equilibrium point $(L_1,L_1+C)$ simplified to
	\[\begin{aligned}
	J\left(L_1,L_1+C\right) &=\begin{pmatrix}
	S\left(L_1+A\right)\left(L_1+C\right) & -S\left(L_1+A\right)\left(L_1+C\right) \\ 
	S\left(L_1+A\right)\left(L_1+C\right) & -S\left(L_1+A\right)\left(L_1+C\right) 
	\end{pmatrix}\\
	&=-\dfrac{1}{4}S(W+A+M+1)(W-A+M+1+2C)\begin{pmatrix}
	1 & -1 \\ 
	1 & -1 
	\end{pmatrix}. 
	\end{aligned}\]
	Now, I find the Jordan normal form of $J\left(L_1,L_1+C\right)$ which has equal eigenvalues and a unique eigenvector
	$\begin{pmatrix}
	1 \\ 
	1 
	\end{pmatrix}$. This vector will be the first column of the matrix of transformations $\Upsilon$. To obtain the second column I choose a vector that makes the matrix $\Upsilon$, that is $\begin{pmatrix}
	-1 \\ 0 \end{pmatrix}$. Thus,
	\[\Upsilon=\begin{pmatrix} 1 & -1 \\ 1 & 0 \end{pmatrix}~\text{and}~\Upsilon^{-1}\left(J\left(L_1,L_1+C\right)\right)\Upsilon=\begin{pmatrix}0 & -\dfrac{1}{4}S(A-W+M+1)(A+W-M-1) \\ 0 & 0 \end{pmatrix}.\]
	Hence, I have the Bogdanov--Takens bifurcation or bifurcation of codimension 2~\cite{xiao2}. Thus, the point $\left(L_1,L_1+C\right)$ is a cusp point.
\end{proof} 

\begin{theo}
Let the system parameters be such that $\Delta=0$~\eqref{eq6} and $M<0$, then system~\eqref{eq3} undergoes a saddle-node bifurcation at the equilibrium point $(L_{1},L_{1}+C)$ (similarly for $(L_{2},L_{2}+C)$, $(L_{3},L_{3}+C)$ and $(L_{4},L_{4}+C)$).
\end{theo}
\begin{proof}
We will proved that the system~\eqref{eq3} has a saddle-node bifurcation at $Q=S(M-W+A+1)/(2(M-W-A+1))$ based on Sotomayor's theorem~\cite{perko}. For $\Delta=0$ the points $P_{1}$ and $P_{2}$ collapse and $W< u_{1} = u_{2} $. Thus, there are two equilibrium points in the first quadrant. Those are $(W,W+C)$ and $(L_{1},L_{2}+C)$, with $L_{1}=(1+M-A+W)/2$. Moreover, setting the dynamical system~\eqref{eq3} by a vector form given by
	\begin{equation} \label{mhtwwae15}
	f(u,v;Q) =\begin{pmatrix}
	(u+A)(1-u)(u-M)-Qv\\ 
	u-v+C
	\end{pmatrix}.
	\end{equation}
It is clear to see that $detJ((L_{1},L_{1}+C))=0$. 
	
Let $V=\begin{pmatrix}v_1 & v_2 \end{pmatrix}^T=\begin{pmatrix}
	1 & 1 \end{pmatrix}$ the eigenvector corresponding to the eigenvalue $\Delta=0$ of the matrix $J(L_{1},L_{1}+C)$. In addition, let $U=\begin{pmatrix}
	u_1 & u_2 \end{pmatrix}^T=\begin{pmatrix}
	\dfrac{S(A+M-W+1)}{Q(A-M+W-1)} & 1 \end{pmatrix}$ the eigenvector corresponding to the eigenvalue $\Delta=0$ of the matrix $(J(L_{1},L_{1}+C))^T$.
	
	On the other hand, differentiating the the vector function~\eqref{mhtwwae15}
	with respect to the bifurcation parameter $Q$ we obtain
	\[f_Q(u,v,Q)=\begin{pmatrix}
	-\dfrac{M-W-A+1+2C}{2}\\
	0
	\end{pmatrix}.\]
	Therefore,
	\[Uf_Q(u,v;Q)=-\dfrac{S(A+M-W+1)(2C-A+M-W+1)}{2Q(A-M+W-1)}\neq0.\]
	Now we analyse the expression $U[D^2f(u,v;Q)(V,V)]$ where $V=(v_1,v_2)$ and \[D^2f(u,v;Q)(V,V)=\begin{pmatrix} -2(A-M+2) \\ 0 \end{pmatrix}.\]
	
	Thus,
	\[\begin{aligned}
	U[D^2f(u,v;Q)]=-\dfrac{2S(A-M+2)(A+M-W+1)}{Q(A-M+W-1)}\neq0
	\end{aligned}.\]
	Therefore, by Sotomayor's theorem the system~\eqref{eq3} has a saddle-node bifurcation at $(L_{1},L_{1}+C)$. It also follow that system~\eqref{eq3} undergoes to a saddle-node bifurcation at the equilibrium points $(L_{2},L_{2}+C)$, $(L_{3},L_{3}+C)$ and $(L_{4},L_{4}+C)$.
\end{proof}

The bifurcation curves divide the parameters plane $(Q,S)$ into six parts. Note that system~\eqref{eq3} can have up to three positives equilibrium points. In order to explain the bifurcation $(W,W+C)<P_{1}<P_{2}$ in which $P_1$ is a saddle point; and the equilibrium points $(W,W+C)$ and $P_2$ can be stable or unstable node. Note that there are two more cases when $P_{1}<(W,W+C)<P_{2}$  with $(W,W+C)$ a saddle point and $P_{1}<P_{2}<(W,W+C)$ where $P_{2}$ now is a saddle point. In addition, when parameters lie in the curve $Q=Q^*$ or $Q=Q^{**}$  the equilibrium points $P_{1}$ and $P_{2}$ collapse and thus system~\eqref{eq3} undergoes to a saddle-node bifurcation and a Bogdanov--Takens bifurcation. Moreover, if the parameters $(Q,S)$ are located in Region I, then system~\eqref{eq3} does one positive equilibrium point in the first quadrant which can be stable or unstable. If the parameters $(Q,S)$ are moved to Regions II, III, IV and V, then system~\eqref{eq3} has two equilibrium points $P_{1}$ which is  a saddle point and $P_{2}$ which is unstable when it is located in Region II, stable when it is in Region III and stable surrounded by an unstable limit cycle when it is in Region IV, while if $(Q,S)$ are located in Region V. In Region V the equilibrium point $(W,W+C)$ is a stable node, $P_{1}$ is a saddle point and $P_{2}$ is a stable node. Finally, if the parameters $(Q,S)$ lie in Region VI, then the equilibrium point $(W,W+C)$ is global stable, see Figure~\ref{Fig13}. Furthermore, when the parameters lie in Region VI system~\eqref{eq3} has three positive equilibrium points.

\begin{figure}
\centering
\includegraphics[width=10cm]{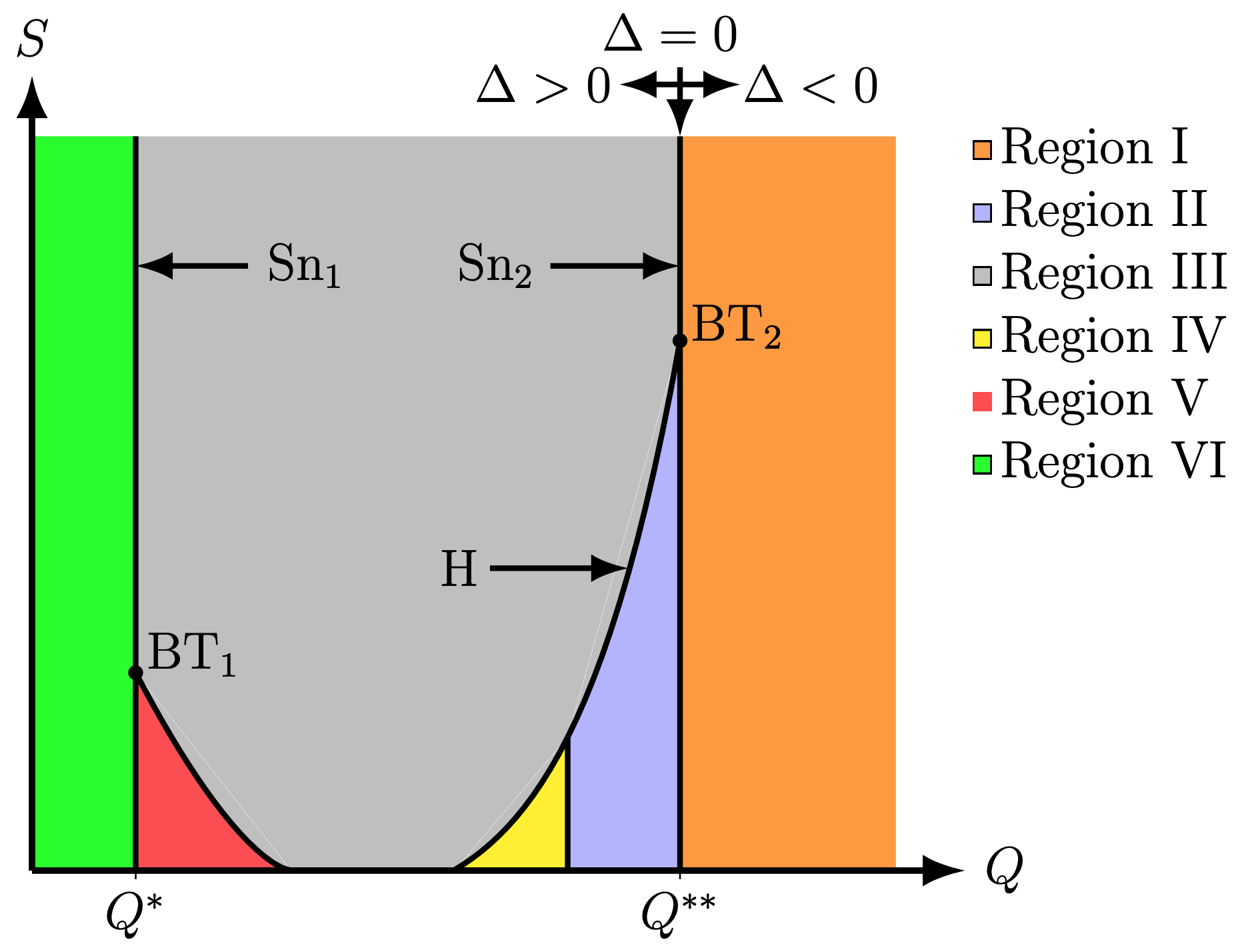}
\caption{The bifurcation diagram of system~\eqref{eq3} with weak Allee effect for $(M,A,C)=(-0.1,0.08,0.19)$ fixed and created with the numerical bifurcation package MATCONT~\cite{matcont}. The curve H represents the Hopf, and Sn$_{1,2}$ represent the saddle-node bifurcation. The point BT$_{1,2}$ represent the Bogdanov--Takens bifurcation.}\label{Fig13}
\end{figure}

\section{Conclusion}\label{S4}

In this manuscript, I study the modified Leslie--Gower predator-prey model with Holling type II functional response, strong (i.e system~\eqref{eq2} with $m>0$) and weak (i.e system~\eqref{eq2} with $m<0$) Allee effect on the prey and a generalist predator. I simplify the analysis by studying a topologically equivalent system~\eqref{eq3} which has four equilibrium points in the axis and up to two positive equilibrium points when a strong Allee effect is included, while system~\eqref{eq3} with weak Allee effect has three equilibrium points in the axis and up to three positive equilibrium points, see Figures~\ref{Fig1} and~\ref{Fig6}. Furthermore, I prove that, when a strong Allee effect is included, system~\eqref{eq3} the equilibrium point $P_1$ is always a saddle point, while $P_2$ can be stable, stable surrounded by an unstable limit cycle or unstable node. Moreover, I also prove that, when a weak Allee effect is included, the equilibrium points $P_1$, $P_2$ and $(W,W+C)$ can be saddle and/or (un)stable points. Besides, when there are three equilibrium points in the first quadrant one of them (the meddle point) is always a saddle point. The stable manifold of the saddle equilibrium point determines a separatrix curve which divides the basins of attraction between the other two equilibrium points. Additionally, I show that system~\eqref{eq3} with weak Allee effect can support a stable limit cycle. 

As the function $\varphi$ is a diffeomorphism preserving the orientation of time, the dynamics of system~\eqref{eq2} is topologically equivalent to system~\eqref{eq3}~\cite{arancibia}. Therefore, I can conclude that when $m>0$ there are conditions in the system parameter for which the predator and prey can coexist or the prey population can extinct. Since the predator population is a generalist specie and thus it avoids extinction by utilising an alternative source of food. Whereas, when $m<0$ there are conditions in the system parameter for which both species can coexist, oscillate or the prey population can extinct when the alternative food ($c$) is bigger than the ratio between the prey intrinsic growth rate ($q$) and the predation rate per capita ($r$) if I assume that the measure of the quality of the prey as food for the predator ($a$) and the Allee threshold ($m<0$) are constant, i.e $c^*<amr/q$.

 I showed that the weak Allee effect and $c^*<amr/q$ in the modified Leslie--Gower model~\eqref{eq2} better represent the dynamics of the original Leslie--Gower predator-prey model studied, for example, by Saez and Gonzalez-Olivares~\cite{saez}. From~\cite{saez}, I can conclude that species in system~\eqref{eq2} could coexist or oscillate but could not extinct. Since there is always one positive equilibrium point which can be stable, or unstable surrounded by a stable limit cycle, or stable surrounded by two limit cycles. 

This manuscript complements the results of the Leslie--Gower model studied by Courchamp \textit{et al}.\ \cite{courchamp2} in which the prey is affected by a density-dependent phenomenon or Allee effect. I showed the impact in the stabilisation, extinction and/or oscillation of the species by considering the Allee effect in the prey for two parameters which are the rescaled intrinsic growth rate of the predator and the predation rate, see Figures~\ref{Fig5} and~\ref{Fig13}. Additionally, I extend the result of Arancibia-Ibarra and Gonz\'alez-Olivares~\cite{arancibia} in which system~\eqref{eq2} was studied partially. I show the impact in the predator and prey interaction by considering a generalist specie and a density-dependent phenomenon together.
 
In summary, the bifurcation diagrams of the modified Leslie--Gower model~\eqref{eq2} with strong Allee effect (see Figure~\ref{Fig5}) and with weak Allee effect (see Figure~\ref{Fig13}) are often qualitatively similar with the bifurcation diagram of the original model~\eqref{eq2} but their solutions behave quantitatively different. In other words, it is observed that the model support equivalent ecological behaviour due to the addition of the modifications into the Leslie--Gower model. That is, a strong Allee effect ($m>0$) and alternative food support coexistence and extinction of the species. In contrast, the modifed model and the model with weak Allee effect ($m<0$) and alternative food does not support the extinction of the species when the density of the alternative food is low.

\bibliographystyle{unsrt}
\bibliography{References}
\end{document}